\documentclass{article}

\DeclareFontFamily{U}{mathx}{\hyphenchar\font45}
\DeclareFontShape{U}{mathx}{m}{n}{
      <5> <6> <7> <8> <9> <10>
      <10.95> <12> <14.4> <17.28> <20.74> <24.88>
      mathx10
      }{}
\DeclareSymbolFont{mathx}{U}{mathx}{m}{n}
\DeclareFontSubstitution{U}{mathx}{m}{n}
\DeclareMathAccent{\widecheck}{0}{mathx}{"71}
\DeclareMathAccent{\wideparen}{0}{mathx}{"75}

\usepackage{mathtools}
\usepackage{amsmath}
\usepackage{amssymb,bbm}
\usepackage{amsfonts}
\usepackage{amsthm}
\usepackage{algorithm}%
\usepackage{algorithmicx}%
\usepackage{algpseudocode}
\usepackage{url}
\usepackage{fancyhdr}
\usepackage[utf8]{inputenc}

\newtheorem{theorem}{Theorem}
\newtheorem{example}{Example}%
\newtheorem{remark}{Remark}%
\newtheorem{lemma}{Lemma}

\newtheorem{definition}{Definition}%
\usepackage{booktabs}

\def\bibcommenthead{}%

\if@Spr@basic@refstyle%
  \usepackage[authoryear]{natbib}%
  \def\bibfont{\reset@font\fontfamily{\rmdefault}\normalsize\selectfont}%
\fi%
\if@Mathphys@refstyle%
  \usepackage[numbers,sort&compress]{natbib}%
  \def\bibfont{\reset@font\fontfamily{\rmdefault}\normalsize\selectfont}%
\fi%
\if@APS@refstyle%
  \usepackage[numbers,sort&compress]{natbib}%
  \def\bibfont{\reset@font\fontfamily{\rmdefault}\normalsize\selectfont}%
\fi%
\if@Vancouver@refstyle%
  \usepackage[numbers,sort&compress]{natbib}%
  \def\bibfont{\reset@font\fontfamily{\rmdefault}\normalsize\selectfont}%
\fi%
\if@APA@refstyle%
  \usepackage[natbibapa]{apacite}%
  \def\refdoi#1{\urlstyle{rm}\url{#1}}%
  \AtBeginDocument{%
  }%
  \def\bibfont{\reset@font\fontfamily{\rmdefault}\normalsize\selectfont}%
\fi%
\if@Chicago@refstyle%
  \usepackage[authoryear]{natbib}%
  \hypersetup{urlcolor=black,colorlinks=false,pdfborder={0 0 0}}
  \def\bibfont{\reset@font\fontfamily{\rmdefault}\normalsize\selectfont}%
\fi%
\if@Standard@Nature@refstyle%
  \usepackage[numbers,sort&compress]{natbib}%
  \def\bibfont{\reset@font\fontfamily{\rmdefault}\normalsize\selectfont}%
\fi%
\if@Default@refstyle%
  \usepackage[numbers,sort&compress]{natbib}%
  \def\bibfont{\reset@font\fontfamily{\rmdefault}\normalsize\selectfont}%
\fi%

\AtBeginDocument{\allowdisplaybreaks}%

\def\eqnheadfont{\reset@font\fontfamily{\rmdefault}\fontsize{16}{18}\bfseries\selectfont}%

\title{Toward Fast and Provably Accurate Near-field Ptychographic Phase Retrieval}
\author{Mark Iwen, Michael Perlmutter, and Mark Philip Roach}

\begin{document}









\maketitle

\abstract{Ptychography is an imaging technique that involves a sample being illuminated by a coherent, localized probe of illumination. When the probe interacts with the sample, the light is diffracted and a diffraction pattern is detected. Then the sample (or probe) is shifted laterally in space to illuminate a new area of the sample whilst ensuring sufficient overlap. Similarly, in Fourier ptychography a sample is illuminated at different angles of incidence (effectively shifting the sample's Fourier transform) after which a lens acts as a low-pass filter, thereby effectively providing localized Fourier information about the sample around frequencies dictated by each angle of illumination.  Mathematically, one therefore obtains a similar set of overlapping measurements of the sample in both Fourier ptychography and ptychography, except in the different domains (Fourier for the former, and physical for the latter). In either case, one is then able to reconstruct an image of the sample from the measurements using similar methods.

\indent Near-Field (Fourier) Ptychography (NFP) (see, e.g., \cite{stockmar2013near,stockmar2015x,zhang2019near}) occurs when the sample is placed at a short defocus distance having a large Fresnel number.
In this paper, we prove that certain NFP measurements are robustly invertible (up to an unavoidable global phase ambiguity) for specific Point Spread Functions (PSFs) and physical masks which lead to well-conditioned lifted linear systems. We then apply a block phase retrieval algorithm using weighted angular synchronization and prove that the proposed approach accurately recovers the measured sample for these specific PSF and mask pairs. Finally, we also propose using a Wirtinger Flow for NFP problems and numerically evaluate that alternate approach both against our main proposed approach, as well as with NFP measurements for which our main approach does not apply.
}

\section{Introduction}\label{sec1}

The task of recovering a complex signal $\mathbf{x}\in\mathbb{C}^d$ from phaseless magnitude measurements is called the \textit{phase retrieval problem}. These types of problems appear in many applications such as optics \cite{antonello2015modal,walther1963question} and x-ray crystallography \cite{buhler2005symmetric,liu2012phase}. Here, we are interested in phase retrieval problems arising from 
(Fourier) ptychography \cite{R2008ptychography,zheng2016fourier}.  Ptychography is an imaging technique involving a sample illuminated by a coherent and often localized probe of illumination.  When the probe interacts with the sample, light is diffracted and a diffraction pattern is detected. 
The probe, or the sample, is then shifted laterally in space to illuminate a new area of the sample while ensuring there is sufficient overlap between each neighboring shift. 
The intensity of the diffraction pattern detected at position $\ell$ resulting from the $k^{\rm th}$ shift of the probe along the sample takes the general form of 
\begin{equation}
\label{equ:GeneralModel}
    \tilde{Y}_{k,\ell} = \lvert (D( S_k {\bf m} \circ {\bf x} ))_\ell \rvert^2,
\end{equation}
where ${\bf x} \in \mathbb{C}^d$ is the sample being imaged, ${\bf m} \in \mathbb{C}^d$ is a mask which represents the probe's incident illumination on (a portion of) the sample, $\circ$ denotes the Hadamard (pointwise) product, $S_k$ is a shift operator, and $D: \mathbb{C}^d \rightarrow \mathbb{C}^d$ is a function that describes the diffraction of the probe radiation from the sample to the plane of the detector after possibly passing through, e.g, a lens.  Similarly, Fourier ptychography ultimately results in the same type of measurements as in \eqref{equ:GeneralModel} except with ${\bf m}$ and ${\bf x}$ replaced by $\widehat{\bf m}$ and $\widehat{\bf x}$, respectively (see, e.g., \cite{zheng2021concept}).

Prior work in the computational mathematics community related to (Fourier) ptychographic imaging has primarily focused on Far-Field\footnote{Far-field versus near-field measurements are defined based on the \textit{Fresnel number} of the imaging system.  See, e.g., \cite{jenkins1957fundamentals} for details.} Ptychography (FFP) in which $D$ is the action of a discrete (inverse) Fourier transform matrix (see, e.g., \cite{qian2014efficient,IVW16,IPSV18,fannjiang2020fixed,preskitt2021admissible,Perlmutter2020}) in \eqref{equ:GeneralModel}.  Here, in contrast, we consider the less well studied setting of near-field ptychography (NFP) which describes situations where, e.g., the masked sample is too close to the source/detector to be well described by the FFP model.  See, e.g., \cite{stockmar2013near,stockmar2015x,zhang2019near} for such imaging applications as well as for more detailed related discussions.  In all of these NFP applications the acquired measurements can again be written in the form of \eqref{equ:GeneralModel} where $D$ is now a convolution operator with a given Point Spread Function (PSF) ${\bf p} \in \mathbbm{C}^d$. 

Let $\mathbf{x} \in \mathbb{C}^d$ denote an unknown  sample,
$\mathbf{m} \in \mathbb{C}^d$ be a known mask, and $\mathbf{p} \in \mathbb{C}^d$ be a known PSF, respectively.  For the remainder of this paper we will suppose we have noisy discretized NFP measurements of the form
\begin{align} \label{eqn: noisy near field}
Y_{k,\ell} = Y_{k,\ell} \left( {\bf x} \right) &\coloneqq \lvert (\mathbf{p} * (S_k \mathbf{m} \circ \mathbf{x}))_\ell \rvert^2 + N_{k,\ell}, \quad (k,\ell) \in \mathcal{S} \subseteq [d]_0 \times [d]_0,
\end{align}
where $S_k$ is a circular shift operator $(S_k\mathbf{x})_n=\mathbf{x}_{n+k ~ \text{mod} ~ d}$, $\mathbf{N} = (N_{k,\ell})$ is an additive noise matrix, and $[d]_0 := \{ 0, \dots, d-1 \}$.  Throughout this paper we will always  index vectors and matrices modulo $d$ unless otherwise stated.

\subsection{Results, Contributions, and Contents}

Our main theorem guarantees the existence of a PSF ${\bf p} \in \mathbb{C}^d$ and a locally supported mask ${\bf m} \in \mathbb{C}^d$ with ${\rm supp}(\mathbf{m}) \subseteq [\delta]_0 \coloneqq \{ 0, \dots, \delta - 1 \}$, $\delta \ll d$, for which the measurements \eqref{eqn: noisy near field} can be inverted up to a global phase factor by a computationally efficient and noise robust algorithm.  In particular, we prove the following result which we believe to be the first theoretical error guarantee for a recovery algorithm in the setting of NFP. 

\begin{theorem}[Inversion of NFP Measurements] \label{thm:MainThm}
Choose $\delta \in [d]_0$ such that $2 \delta - 1$ divides $d$. Then, there exists a PSF ${\bf p} \in \mathbb{C}^d$ and a mask ${\bf m} \in \mathbb{C}^d$ with ${\rm supp}(\mathbf{m}) \subseteq [\delta]_0$ such that Algorithm~\ref{NFP-BlockPR} below, when provided with input measurements \eqref{eqn: noisy near field}, will return an estimate $\mathbf{x}_{\text{est}} \in \mathbb{C}^d$ of ${\bf x}$ 
satisfying
\begin{align*}
\underset{\phi \in [0,2\pi)}{\min} \|\mathbf{x}_{\text{est}} - e^{\mathbbm{i}\phi}\mathbf{x}\|_{2} \leq C \bigg( \|\mathbf{x}\|_{\infty}\dfrac{d\sqrt{\delta} \sqrt{\|\mathbf{x}_{\text{est}}\|_{\infty}^2 + \|\mathbf{x}_{\text{est}}\|_{\infty}^3}}{\lvert\mathbf{x}_{\text{est}}\rvert_{\min}^2} \cdot  \|\mathbf{N}\|_{\rm F} + \sqrt{d \delta  \|\mathbf{N}\|_{\rm F}}\bigg).
\end{align*}
Here $C \in \mathbb{R}^+$ is an absolute constant\footnotemark, and $\lvert\mathbf{x}_{\text{est}}\rvert_{\min}$ denotes the smallest magnitude of any entry in $\mathbf{x}_{\text{est}}$.
\end{theorem}
\footnotetext{In this paper we will use $C$ to denote absolute constants which may change from line to line.}
Looking at Theorem~\ref{thm:MainThm} we can see, e.g., that in the noiseless setting where $\|\mathbf{N}\|_{\rm F} = 0$ the output $\mathbf{x}_{\text{est}}$ of Algorithm~\ref{NFP-BlockPR} is guaranteed to match the measured signal ${\bf x}$ up to a global phase factor whenever $\mathbf{x}_{\text{est}}$ has no zeros.\footnote{Note that prior work on far-field ptychography assumed that ${\bf x}$ itself was non-vanishing (see e.g. \cite{IVW16,IPSV18}).  However, requiring $\mathbf{x}_\text{est}$ to not vanish is  more easily verifiable in practice.}  Moreover, the method is also robust to small amounts of additive noise.  The proof of Theorem~\ref{thm:MainThm} consists of two parts:  First, in Section~\ref{sec:NearfromFar}, we show that a specific PSF and mask choice results in NFP measurements \eqref{eqn: noisy near field} which are essentially equivalent to far-field ptychographic measurements \eqref{equ:FFPmeasurements} that are known to be robustly invertible by prior work \cite{IVW16,IPSV18,preskitt2021admissible}.   This guarantees the existence of PSFs and masks which allow for the robust inversion of \eqref{eqn: noisy near field} up to a global phase. However, these prior works all prove error bounds on $\underset{\phi \in [0,2\pi)}{\min} \|\mathbf{x}_{\text{est}} - e^{\mathbbm{i}\phi}\mathbf{x}\|_{2}$ which scale {\it quadratically} in $d$ (see, e.g., Corollary 3 in \cite{IPSV18} and Theorem 1 in \cite{preskitt2021admissible}).  This motivates the second part of the proof in  Section~\ref{sec:ImprovedBlockPRanalysis}, where we improve these results so that they only depend {\it linearly} on $d$.  This is achieved by utilizing weighted angular synchronization error bounds from \cite{filbir2021recovery} which require, among other things, updated lower bounds for the second smallest eigenvalue of the unnormalized graph Laplacian of a given weighted graph obtained from $\mathbf{x}$ (derived in Section~\ref{sec:ImprovedBlockPRanalysis} with the help of auxiliary results proven in Appendix~\ref{secB}). 
We also note that  the improved dependence on $d$ proven in Section~\ref{sec:ImprovedBlockPRanalysis} for the FFP methods previously analyzed in \cite{IVW16,IPSV18,preskitt2021admissible} may be of potential independent interest.

Theorem~\ref{thm:MainThm} is proven for a specific $(2 \delta -1)$-periodic PSF $\mathbf{p}$ and locally supported mask $\mathbf{m}$ whose induced lifted linear measurement operator (see \eqref{eqn: block-circulant matrix M} -- \eqref{eqn: Y equation} below together with Lemma~\ref{lem: rewrite}) is provably well conditioned.  See Lemma \ref{lem: Choice of p and m} for the definition of this particular $\mathbf{p},\mathbf{m}$ pair as well as for their measurements' related condition number bound.  However, we note that Algorithm~\ref{NFP-BlockPR} is guaranteed to work well much more generally for {\it any} PSF and mask pair which leads to well-conditioned measurements (up to, at worst, potentially having to change the shift and frequency pairs one samples if, e.g., $\mathbf{p}$ is not periodic -- see Remark~\ref{rmk: full rewritten}).  Indeed, inspecting the proof of Theorem \ref{thm:MainThm} we see that Lemma \ref{lem: Full Error Bound} decomposes the total error of Algorithm~\ref{NFP-BlockPR}, $\min_{\phi \in [0,2\pi)} \hspace{-.05in} \left\|\mathbf{x} - e^{\mathbbm{i}\phi}\mathbf{x}_{\text{est}}\right\|_{2},$ into terms involving the phase error, $\min_{\phi \in [0,2\pi)} \hspace{-.03in} \|\mathbf{x}^{(\theta)}_{\text{est}} - e^{\mathbbm{i}\phi}\mathbf{x}^{(\theta)} \|_{2}$, and the magnitute error, $\| \mathbf{x}^{(\text{mag})} -  \mathbf{x}^{(\text{mag})}_{\text{est}} \|_{2}$.  The phase error is controlled by Theorem \ref{thm: Improved Phase Bound} and Lemma \ref{lem: Frobenius Bound}. The proof of these results only depends on the choice of $\mathbf{m}$ and $\mathbf{p}$ through $\sigma_{\rm min}(\widecheck{\mathbf{M}}^{(p,m)})$, the minimal singular value of their induced measurement operator. Similarly, the magnitude error is controlled by Lemma \ref{Mag Diff Error} and Theorem \ref{thm: Particular Full Error Bound} which also only depend on $\mathbf{m}$ and $\mathbf{p}$ through $\sigma_{\rm min}(\widecheck{\mathbf{M}}^{(p,m)})$. Therefore, variants of these results can be derived for any invertible measurement system.   Moreover, numerical experiments demonstrate that the proposed method works well for a wide variety of non-vanishing PSF and locally supported mask pairs.  

In order to be able to handle even more general PSFs $\mathbf{p}$ which do however, e.g., vanish, in Section \ref{sec:WTF4NFP} we also propose a Wirtinger Flow based algorithm, Algorithm \ref{NFP Wirtinger Flow}, for inverting NFP measurements \eqref{eqn: noisy near field}.   Though slower than Algorithm~\ref{NFP-BlockPR} and less well supported by theory for the PSF and mask pairs for which both methods work empirically, Algorithm \ref{NFP Wirtinger Flow} generally appears more flexible and, e.g., also requires fewer shifts than Algorithm \ref{NFP-BlockPR} to work well in practice when a given mask is not locally supported.
 Similar to Algorithm \ref{NFP-BlockPR}, Algorithm \ref{NFP Wirtinger Flow} relies on the observation that the NFP measurements \eqref{eqn: noisy near field} are essentially equivalent to FFP measurements as shown in Section \ref{sec:NearfromFar}.  
 In Section~\ref{sec: Numerics}, we evaluate Algorithm~\ref{NFP-BlockPR} and Algorithm \ref{NFP Wirtinger Flow}, 
  numerically, both individually and in comparison to one another in the case of locally supported masks.  Finally, in Section~\ref{sec:Conc}, we conclude with a brief discussion of future work.

\section{Preliminaries:  Prior Results for Far-Field Ptychography using Local Measurements}
\label{sec:Preliminaries}
\begin{table}[htp]
\renewcommand\arraystretch{1.3}
\begin{center}
\begin{minipage}{\textwidth}
\caption{Notational Reference Table}\label{Notational Reference Table}%
\begin{tabular}{@{}lll@{}}
\toprule
\textbf{Notation} & \textbf{Definition} & \textbf{Notes}\\
\hline
$[n]_0$ & $[n]_0 = \{0,1,2, \ldots, n-1\}$ & Zero indexing\\
\hline
$(\mathbf{x})_n$ & $\mathbf{x} \in \mathbb{C}^d, (\mathbf{x})_n = x_{n \; \text{mod} \; d}$ & Vector circular indexing\\
\hline
$(\mathbf{A})_{i,j}$ & $\mathbf{A} \in \mathbb{C}^{m \times n}, (\mathbf{A})_{i,j} = A_{i \; \text{mod} \; m, j \; \text{mod} \; n}$ & Matrix circular indexing\\
\hline
$\langle \mathbf{x}, \mathbf{y} \rangle$ & $\langle \mathbf{x}, \mathbf{y} \rangle = \sum_{n=0}^{d-1} x_n \overline{y}_n = \mathbf{y}^*\mathbf{x}$ & Complex inner product\\
\hline
$\text{supp}(\mathbf{x})$ & $\text{supp}(\mathbf{x}) = \{n \in [d]_0 \mid x_n \neq 0\}$ & Support\\
\hline
$\mathbf{F}_d$ & $(\mathbf{F}_d)_{j,k} = e^{-2\pi \mathbbm{i} j k/d}, \forall (j,k) \in [d]_0 \times [d]_0$ & Discrete Fourier transform matrix\\
\hline
$\widehat{\mathbf{x}}$ & $\widehat{x}_n = (\mathbf{F}_d \mathbf{x})_n = \sum_{k=0}^{d-1} x_k e^{-2\pi \mathbbm{i} nk/d}$ & Discrete Fourier transform\\
\hline
$\mathbf{F}_{d}^{-1} \mathbf{x}$ & $ (\mathbf{F}_{d}^{-1} \mathbf{x})_n = \frac{1}{d} \sum_{k=0}^{d-1} x_k e^{2\pi \mathbbm{i} k n/d}$ & Discrete inverse Fourier transform\\
\hline
$S_k(\mathbf{x})$ & $(S_k \mathbf{x})_n = x_{(n + k) \; {\rm mod} \; d}, \quad \forall n \in [d]_0$ & Circular shift\\
\hline
$\widetilde{\mathbf{x}}$ & $\widetilde{x}_n = x_{-n \; {\rm mod} \; d}, \quad \forall n \in [d]_0$ & Reversal\\
\hline
$\mathbf{x} * \mathbf{y}$ & $(\mathbf{x} * \mathbf{y})_n = \sum_{k=0}^{d-1} x_k y_{n-k}$ & Circular convolution\\
\hline
$\mathbf{x} \circ \mathbf{y}$ & $(\mathbf{x} \circ \mathbf{y})_n = x_n y_n$ & Hadamard (pointwise) product\\
\hline
\end{tabular}
\end{minipage}
\end{center}
\end{table}

Our method, described in  Algorithm~\ref{NFP-BlockPR}, 
is based on relating the near-field ptychographic measurements \eqref{eqn: noisy near field} to far-field ptychographic measurements of the form
\begin{equation}\label{eqn: ffp measurements}
\widetilde{Y}_{k,\ell} = \widetilde{Y}_{k,\ell}\left({\bf x} \right) \coloneqq \left\lvert \sum_{n=0}^{d-1} m'_n x_{n+k} e^{-2\pi \mathbbm{i} \ell n/d}\right\rvert^2 + N_{k,\ell},
\end{equation}
where $\mathbf{m}'$ is a compactly supported mask. 
  If we let $(\widecheck{\mathbf{m}}_\ell)_n \coloneqq \overline{m'_n} e^{-2\pi \mathbbm{i} \ell n/d}$, then these measurements can be written as 
\begin{align}  \label{equ:FFPmeasurements}
\widetilde{Y}_{k,\ell} = \lvert\langle \widecheck{\mathbf{m}}_{\ell} , S_{k}\mathbf{x} \rangle\rvert^2 + N_{k,\ell},
\end{align}
where as above $S_k$ denotes a circular shift of length $k$, i.e., $(S_k\mathbf{x})_n=x_{(n+k)~{\rm mod}~d}$.  In \cite{IVW16}, phase retrieval measurements of this form 
are studied when $\mathbf{m}'$ is supported in an interval of length $\delta$ for some $\delta\ll d$. The fast phase retrieval (fpr) method used there relies on using a 
 lifted linear system involving a block-circulant matrix to recover a portion of the autocorrelation matrix $\mathbf{x}\mathbf{x}^*$. Specifically, letting $D \coloneqq d(2\delta-1)$, the authors define a block-circulant matrix $\widecheck{\mathbf{M}} \in \mathbb{C}^{D \times D}$ by 
\begin{align} \label{eqn: block-circulant matrix M}
\widecheck{\mathbf{M}} \coloneqq
\begin{pmatrix}
\widecheck{\mathbf{M}}_0 & \widecheck{\mathbf{M}}_1 & \dots & \widecheck{\mathbf{M}}_{\delta-1} & 0 & 0 & \dots & 0\\
0 & \widecheck{\mathbf{M}}_0  & \widecheck{\mathbf{M}}_1 & \dots & \widecheck{\mathbf{M}}_{\delta-1} & 0 & \dots & 0\\
\vdots & \vdots & \vdots & \ddots & \ddots & \vdots & \vdots & \vdots\\
\widecheck{\mathbf{M}}_1 & \dots & \widecheck{\mathbf{M}}_{\delta-1} & 0 & 0 & 0 & \dots & \widecheck{\mathbf{M}}_0\\
\end{pmatrix}.
\end{align}
where the matrices $\widecheck{\mathbf{M}}_k \in \mathbb{C}^{(2\delta - 1) \times (2\delta - 1)}$ are defined entry-wise by 
\begin{align}\label{eqn: Mcheck_k} 
(\widecheck{\mathbf{M}}_k)_{\ell j} \coloneqq
\begin{cases}
(\widecheck{\mathbf{m}}_\ell)_{k} \overline{(\widecheck{\mathbf{m}}_\ell)}_{j + k}, \quad &0 \leq j \leq \delta - k,\cr
(\widecheck{\mathbf{m}}_\ell)_{k} \overline{(\widecheck{\mathbf{m}}_\ell)}_{j + k - 2\delta}, \quad &2\delta - 1 + k \leq j \leq 2\delta - 2 \;  \text{and} \; k < \delta,\cr
0, \quad &\text{otherwise.}\cr
\end{cases}
\end{align}
Letting $\mathbf{z}\in\mathbb{C}^d$ be a vector obtained by subsampling appropriate entries of  $\text{vec}(\mathbf{x}\mathbf{x}^*)$, the authors show that, in the noiseless setting, \begin{equation}\label{eqn: Y equation}\text{vec}(\widetilde{\mathbf{Y}}) = \widecheck{\mathbf{M}}\mathbf{z}, \quad \widetilde{\mathbf{Y}} \in \mathbb{C}^{d \times (2\delta-1)}.\end{equation}
(See Equation (9) of \cite{IVW16} for explicit details on the arrangement of the entries.)
%
For properly chosen $\mathbf{m}$, the matrix $\widecheck{\mathbf{M}}$ is invertible, and therefore one may solve for $\mathbf{z}$ by multiplying by $\widecheck{\mathbf{M}}^{-1}$, i.e., $\mathbf{z} = \widecheck{\mathbf{M}}^{-1}\text{vec}(\mathbf{Y})$. 
Then, one may reshape $\mathbf{z}$ to recover a $d\times d$ matrix  $\widehat{\mathbf{X}}$ whose non-zero entries are estimates of the autocorrelation matrix $\mathbf{x}\mathbf{x}^*$. One may then obtain a vector $\mathbf{x}_{\text{est}}$ which approximates  $\mathbf{x}$ by  angular synchronization procedure such as the eigenvector-based method which we will discuss in Section \ref{sec: Eigenvector-Based Angular Synchronization}.\\
\\
In \cite{IVW16}, it is shown that exponential masks $\widecheck{\mathbf{m}}_{\ell}^{(\text{fpr})}$ defined by
\begin{align} \label{equ:Goodmasks}
(\widecheck{\mathbf{m}}_{\ell}^{(\text{fpr})})_n = 
\begin{cases}
\frac{e^{-(n+1)/a}}{\sqrt[\leftroot{-2}\uproot{2}4]{2\delta - 1}} \cdot e^{\frac{2\pi \mathbbm{i} n\ell}{2\delta - 1}}, &n \in [\delta]_0\cr
0, \quad &\text{otherwise}\cr
\end{cases},
 \quad a\coloneqq \max\bigg\{ 4, \dfrac{\delta - 1}{2}\bigg\},
\end{align}
lead to a lifted linear system which is well-conditioned and thus to provable  recovery guarantees for the method described above. In particular, we may obtain the following upper bound for the condition number of  block-circulant matrix $\widecheck{\mathbf{M}}^{(\text{fpr})}$ obtained when one sets $\widecheck{\mathbf{m}}_\ell= \widecheck{\mathbf{m}}_{\ell}^{(\text{fpr})}$. 

\begin{theorem}[Theorem 4 and Equation (33) in \cite{IVW16}]\label{Condition Theorem} The condition number of  $\widecheck{\mathbf{M}}^{(\text{fpr})}$, the matrix obtained by setting  $\widecheck{\mathbf{m}}_\ell =\widecheck{\mathbf{m}}_{\ell}^{(\text{fpr})}$ in \eqref{eqn: Mcheck_k}, 
may be bounded by
\begin{align*}
\kappa\left(\widecheck{\mathbf{M}}^{(\text{fpr})}\right) < \max\bigg\{ 144e^2, \dfrac{9e^2 (\delta - 1)^2}{4}\bigg\} \leq C\delta^2, \quad C \in \mathbb{R}^+.
\end{align*}
Furthermore, $\widecheck{\mathbf{M}}^{(\text{fpr})}$ can be inverted in $\mathcal{O} (\delta \cdot d \log d)$-time and its smallest singular value $\sigma_{\rm min}\left(\widecheck{\mathbf{M}}^{(\textrm{fpr})}\right)$ is bounded from below by $C/\delta$.
\end{theorem}
\subsection{Angular Synchronization}\label{sec: Eigenvector-Based Angular Synchronization}
Inverting $\widecheck{\mathbf{M}}$ as described in the previous subsection allows one to obtain a portion of the autocorrelation matrix $\mathbf{x}\mathbf{x}^*$. This motivates us to consider
angular synchronization, the process of recovering a vector $\mathbf{x}$ from  (a portion of) its autocorrelation matrix $\mathbf{x}\mathbf{x}^*$ (or an estimate  $\widehat{\mathbf{X}}$). One  popular approach, which we discuss below, is based on upon first entry-wise normalizing this matrix and then taking the lead eigenvector. 
Specifically, 
we define a truncated autocorrelation matrix $\mathbf{X}$ corresponding to the true signal $\mathbf{x}$ by
\begin{align} \label{eqn: bigX}
X_{j,k} = 
\begin{cases}
x_{j}\overline{x_k},&\lvert j - k\rvert  \; \text{mod} \; d < \delta\cr
0, \qquad &\text{otherwise.}\cr
\end{cases}
\end{align}
We also define a truncated autocorrelation matrix $\widehat{\mathbf{X}}$ corresponding to our estimate, $\mathbf{x}_{\text{est}}$, given by 
\begin{align}\label{eqn: bigXhat}
 \widehat{X}_{j,k} = 
\begin{cases}
(x_{\text{est}})_{j}\overline{(x_{\text{est}})_k}, &\lvert j - k\rvert  \; \text{mod} \; d < \delta\cr
0, \qquad &\text{otherwise.}\cr
\end{cases}
\end{align}
The method  from \cite{IVW16} is based upon first solving for $\widehat{\mathbf{X}}$ and then solving for $\mathbf{x}_{\text{est}}$. If $\widehat{\mathbf{X}}$ is a good approximation of $\mathbf{X}$, then the results proved in \cite{viswanathan2015fast} show that $\mathbf{x}_{\text{est}}$ will be a good approximation of $\mathbf{x}$.

Moving forward, prior works \cite{IVW16,viswanathan2015fast} effectively decomposed $\mathbf{X} = \mathbf{X}^{(\theta)} \circ \mathbf{X}^{(\text{mag})}$ into its phase and magnitude matrices by setting $X^{(\text{mag})}_{j,k}=\lvert X_{j,k}\rvert$ and $X^{(\theta)}_{j,k}= X_{j,k}/\lvert X_{j,k}\rvert$ if $\lvert X_{j,k}\rvert\neq 0$ with $X^{(\theta)}_{j,k}=0$ otherwise.  One may then write $\widehat{\mathbf{X}} = \widehat{\mathbf{X}}^{(\theta)} \circ \widehat{\mathbf{X}}^{(\text{mag})}$. Note that by construction, if $\mathbf{x}$ is nonvanishing, then we have  $\lvert X^{(\theta)}_{j,k}\rvert = 1$ and $X^{(\text{mag})}_{j,k} > 0$ whenever $\lvert j - k\rvert  \; \text{mod} \; d < \delta$. 
 Letting $\mathbf{u}\in\mathbb{C}^d$ be the leading eigenvector of $\widehat{\mathbf{X}}$ and letting $\text{diag}(\widehat{\mathbf{X}})\in\mathbb{C}^d$ be the main diagonal of $\widehat{\mathbf{X}}$, the output of the resulting algorithm is then $\mathbf{x}_{est} \coloneqq \text{diag}(\widehat{\mathbf{X}}) \circ \mathbf{u}$. 
 \begin{example} \label{eqn: angular sync example}  Let $d = 4, \delta = 2$. Then $\widehat{\mathbf{X}}$ defined as in \eqref{eqn: bigXhat} is given by
\begin{align*} \widehat{\mathbf{X}} 
 = 
\begin{pmatrix}
\lvert{(x_{\text{est}})}_0\rvert^2 & (x_{\text{est}})_0 \overline{(x_{\text{est}})_1} & 0 & (x_{\text{est}})_0 \overline{(x_{\text{est}})_3}&\\
(x_{\text{est}})_1 \overline{(x_{\text{est}})_1} & \lvert (x_{\text{est}})_1\rvert^2 & (x_{\text{est}})_1 \overline{(x_{\text{est}})_2} & 0 &\\
0 & (x_{\text{est}})_2 \overline{(x_{\text{est}})_1} & \lvert(x_{\text{est}})_2\rvert^2  &  (x_{\text{est}})_2 \overline{(x_{\text{est}})_3}&\\
(x_{\text{est}})_3 \overline{(x_{\text{est}})_0}& 0 & (x_{\text{est}})_3 \overline{(x_{\text{est}})_2} & \lvert(x_{\text{est}})_3\rvert^2&\\
\end{pmatrix}.
\end{align*}
If we write $(x_{\text{est}})_n=\lvert(x_{\text{est}})_n\rvert e^{\mathbbm{i}\theta_n}$, then we may compute
\begin{align*} \widehat{\mathbf{X}}^{(\theta)} 
=
\begin{pmatrix}
1 & e^{i(\theta_0 - \theta_1)} & 0 &  e^{i(\theta_0 - \theta_3)}&\\
e^{i(\theta_1 - \theta_0)}& 1 & e^{i(\theta_1 - \theta_2)} & 0 &\\
0 & e^{i(\theta_2 - \theta_1)}& 1  &  e^{i(\theta_2 - \theta_3)}&\\
e^{i(\theta_3 - \theta_0)} & 0 & e^{i(\theta_3 - \theta_2)} & 1&\\
\end{pmatrix}.
\end{align*}
One may verify that the lead eigenvector is $\mathbf{u} = (e^{\mathbbm{i}\theta_0} \; e^{i \theta_1} \; e^{i \theta_2} \; e^{\mathbbm{i}\theta_3})^T$ and therefore 
\begin{align*} 
\mathbf{x}_{\rm{est}} = \sqrt{\rm{diag}(\widehat{\mathbf{X}})} \circ \mathbf{u} = (\lvert(x_{\text{est}})_0\rvert e^{\mathbbm{i}\theta_0} \; \lvert(x_{\text{est}})_1\rvert e^{\mathbbm{i}\theta_1} \; \lvert(x_{\text{est}})_2\rvert e^{\mathbbm{i}\theta_2} \; \lvert(x_{\text{est}})_3\rvert e^{\mathbbm{i}\theta_3})^T.
\end{align*}
\end{example}
In Section \ref{sec:ImprovedBlockPRanalysis}, we will discuss another slightly more sophisticated way for estimating the phases based on Algorithm 3 of \cite{preskitt2018phase} which involves taking the smallest eigenvector of an appropriately weighted graph Laplacian.  Indeed, this new angular synchronization approach is what ultimately allows for the NFP error bound in Theorem~\ref{thm:MainThm} to have improved dependence on signal dimension $d$ over prior FFP error bounds in \cite{IVW16,IPSV18,preskitt2021admissible}.  The end result will be a more accurate method for computing $\widehat{\mathbf{X}}$ in \eqref{eqn: bigXhat} from given NFP measurements \eqref{eqn: noisy near field}.

\section{Near from Far:  Guaranteed Near-Field Ptychographic Recovery via Far-Field Results}
\label{sec:NearfromFar}
In this section, we show how to relate the near-field ptychographic measurements \eqref{eqn: noisy near field} to the far-field ptychographic measurements \eqref{equ:FFPmeasurements}. This will allow us to recover $\mathbf{x}$ by using methods similar to those introduced in \cite{IVW16}. In order get nontrivial bounds, we  will also need to prove the existence of an admissible PSF and mask pair, ${\bf p} \in \mathbb{C}^d$ and  ${\bf m} \in \mathbb{C}^d$, which lead to a well conditioned linear system in \eqref{eqn: Y equation}.  In particular,  we will present a PSF and mask pair  
such that the resulting block-circulant matrix, denoted $\widecheck{\mathbf{M}}^{(p,m)}$, will have the same condition number as the matrix
$\widecheck{\mathbf{M}}^{(\text{fpr})}$ constructed from the masks $\widecheck{\mathbf{m}}_\ell^{(\text{fpr})}$ defined in \eqref{equ:Goodmasks}. Therefore, Theorem~\ref{Condition Theorem} will allow us to obtain convergence guarantees for Algorithm \ref{NFP-BlockPR}.  

Here, we will set the measurement index set $\mathcal{S}$ considered in \eqref{eqn: noisy near field} to be $\mathcal{S} = \mathcal{K} \times \mathcal{L}$ where $\mathcal{K}=[d]_0$ and $\mathcal{L}=[2\delta-1]_0$.
The following lemma proves that we can rewrite NFP measurements from \eqref{eqn: noisy near field} as local FFP measurements of the form \eqref{equ:FFPmeasurements} as long as the mask ${\bf m}$ has local support and the PSF is periodic.  
It will be based upon defining masks
\begin{align} \label{equ:NFPmaskdef}
\widecheck{\mathbf{m}}_{\ell}^{(p,m)}\coloneqq \overline{S_{\ell}\widetilde{\mathbf{p}} \circ \mathbf{m}} \in \mathbb{C}^d, 
\end{align}
where  $\widetilde{\mathbf{p}}$ is the reversal of $\mathbf{p}$ about its first entry modulo $d$, i.e., $\widetilde{p}_n=p_{-n \text{ mod } d}$ 
Since the masks $\widecheck{\mathbf{m}}_{\ell}^{(p,m)}$ have compact support, this will then yield a lifted set of linear measurements of the type considered in \cite{IVW16,IPSV18,preskitt2021admissible}. 

\begin{lemma}\label{lem: rewrite}
Let $\mathcal{S}=\mathcal{K}\times\mathcal{L}= [d]_0\times[2\delta-1]_0,$ and recall the measurements 
\begin{align*} 
Y_{k,\ell} = \lvert (\mathbf{p} * (S_k \mathbf{m} \circ \mathbf{x}))_\ell \rvert^2, \; (k,\ell) \in \mathcal{S},
\end{align*}
defined in \eqref{eqn: noisy near field}.
Suppose that $2 \delta - 1$ divides $d$, that ${\bf p} \in \mathbb{C}^d$ is $2 \delta - 1$ periodic, and that  ${\bf m} \in \mathbb{C}^d$ satisfies ${\rm supp}({\bf m}) \subseteq [\delta]_0$.
Then,  we may rearrange   the measurements \eqref{eqn: noisy near field} into a matrix of  FFP-type measurements 
\begin{align} \label{equ:newFFPmeasurements}
\widetilde{Y}_{k,\ell} \coloneqq Y_{-k~{\rm mod }~d,~k-\ell ~{\rm mod}~2\delta-1} = \lvert \langle \widecheck{\mathbf{m}}_{\ell}^{(p,m)} , S_{k}\mathbf{x} \rangle \rvert ^2, ~ (k,\ell) \in [d]_0 \times [2\delta-1]_0,
\end{align}
where $\widecheck{\mathbf{m}}_{\ell}^{(p,m)}$ is defined as in \eqref{equ:NFPmaskdef}.
As a consequence, recovering $\mathbf{x}$ is equivalent to inverting a block-circulant matrix as described in \eqref{eqn: block-circulant matrix M} -- \eqref{eqn: Y equation}.
\end{lemma}

\begin{proof} By Lemma \ref{Shift Identities} part \ref{Shift Lemma 1} , Lemma \ref{Lemma 2}, Lemma \ref{Shift Identities} part \ref{Shift Lemma 2}, and Lemma \ref{Conjugation Identity} from Appendix \ref{secA}, we have that
\begin{align*}\nonumber
Y_{k,\ell} = \lvert (\mathbf{p} * (S_k \mathbf{m} \circ \mathbf{x}))_\ell\rvert^2 &= \lvert\langle S_{-\ell} \widetilde{\mathbf{p}}, \overline{S_k \mathbf{m} \circ \mathbf{x}} \rangle\rvert^2&\\
&= \lvert\langle S_{-\ell} \widetilde{\mathbf{p}} \circ S_k \mathbf{m}, \overline{\mathbf{x}} \rangle\rvert^2&\\\nonumber
&= \lvert\langle S_k (S_{-\ell - k}\widetilde{\mathbf{p}} \circ \mathbf{m}), \overline{\mathbf{x}} \rangle\rvert^2&\\\nonumber
&= \lvert\langle S_k (\overline{S_{-\ell - k}\widetilde{\mathbf{p}} \circ \mathbf{m}}), \mathbf{x} \rangle \rvert^2\\\nonumber
&= \lvert\langle S_k (\overline{S_{-\ell - k ~{\rm mod }~2\delta - 1}\widetilde{\mathbf{p}} \circ \mathbf{m}}), \mathbf{x} \rangle \rvert^2,
\end{align*}
where the last equality uses the fact that ${\bf p}$ is $2 \delta - 1$ periodic.
 We may now apply Lemma \ref{Shift Identities} part \ref{Shift Lemma 3} to see that\begin{align*}\nonumber
Y_{k,\ell} = \lvert\langle S_k (\overline{S_{-\ell - k ~{\rm mod }~2\delta - 1}\widetilde{\mathbf{p}} \circ \mathbf{m}}), \mathbf{x} \rangle \rvert^2 = \lvert\langle  (\overline{S_{-\ell - k ~{\rm mod }~2\delta - 1}\widetilde{\mathbf{p}} \circ \mathbf{m}}), S_{-k}\mathbf{x} \rangle \rvert^2 .
\end{align*}
Finally, since $\widecheck{\mathbf{m}}_{\ell}^{(p,m)} = \overline{S_{\ell}\widetilde{\mathbf{p}} \circ \mathbf{m}}$, we see that for all $k\in[d]_0$ and all $\ell\in[2\delta-1]_0$, we have 
\begin{align*}
    \widetilde{Y}_{k,\ell} &= Y_{-k~{\rm mod }~d,~k-\ell ~{\rm mod}~2\delta-1}\\
    &=\lvert\langle  (\overline{S_{-(k-\ell) - (-k) ~{\rm mod }~2\delta - 1}\widetilde{\mathbf{p}} \circ \mathbf{m}}), S_{-(-k)}\mathbf{x} \rangle \rvert^2\\
    &=\lvert\langle  (\overline{S_{\ell ~{\rm mod }~2\delta - 1}\widetilde{\mathbf{p}} \circ \mathbf{m}}), S_{k}\mathbf{x} \rangle \rvert^2\\
    &=\lvert \langle \widecheck{\mathbf{m}}_{\ell}^{(p,m)} , S_{k}\mathbf{x} \rangle \rvert ^2.
\end{align*} 
\vspace{-5mm}{}\end{proof}

\begin{remark} \label{rmk: full rewritten} If we instead change the pairs $\mathcal{S}$ in Lemma~\ref{lem: rewrite} for which we collect NFP measurements \eqref{eqn: noisy near field} to be $\mathcal{S}':= \cup_{k \in [d]_0} \{d-k\} \times \{k-2\delta+2, \dots, k-1,k\} \mod d$, then we may remove the assumption that $\mathbf{p}$ is $2\delta-1$ periodic. In particular, for $(k,\ell) \in \mathcal{S}'$ one may substitute $k=d-k'$ and $\ell'= k'-i$ for some $0\leq i\leq 2\delta-2$ to see that then $(-k'-\ell') \text{ mod } d = i$. Thus, since $0\leq i \leq 2\delta-2$, $(-\ell'-k')\text{ mod } 2\delta-1=(-\ell'-k')\text{ mod } d$,
and so we may use the same calculation as above 
without assuming that $\mathbf{p}$ is $2 \delta - 1$ periodic.  Note, however, that $\mathcal{S}$ has a simple Cartesian product structure whereas $\mathcal{S}'$ does not.  As a result, the entries of $\mathbf{p} * (S_k \mathbf{m} \circ \mathbf{x})$ that one must sample varies based on the mask shift $k$ in the case of $\mathcal{S}'$, potentially complicating the collection of the associated NFP measurements \eqref{eqn: noisy near field} in some situations.
\end{remark}

Next, in Lemma \ref{lem: Choice of p and m} below, we will show how to choose a mask $\mathbf{m}$ and PSF $\mathbf{p}$
such that  $\widecheck{\mathbf{m}}_{\ell}^{(p,m)}$ defined as in \eqref{equ:NFPmaskdef} and  $\widecheck{\mathbf{m}}_{{2 \ell  \; \text{mod} \; 2\delta-1} }^{(\text{fpr})}$ defined as in \eqref{equ:Goodmasks} will only differ by a global phase for each $\ell \in [2 \delta - 1]_0$.   As a consequence, we obtain the desired result that the block-circulant matrix arising from the NFP measurements \eqref{eqn: noisy near field} is essentially equivalent (up to a row permutation and global phase shift) to the well-conditioned lifted linear measurement operator $\widecheck{\mathbf{M}}^{(\text{fpr})}$ considered in Theorem \ref{Condition Theorem}.

\begin{lemma}\label{lem: Choice of p and m} Let $\mathbf{p},\mathbf{m}\in\mathbb{C}^d$ have entries given by $$p_{n} \coloneqq e^{-\tfrac{2 \pi \mathbbm{i} n^2}{2\delta-1}}, \quad\text{and}\quad m_n \coloneqq 
\begin{cases}
\dfrac{e^{-n+1}/a}{\sqrt[\leftroot{-2}\uproot{2}4]{2\delta - 1}}\cdot e^{\tfrac{2\pi \mathbbm{i} n^2}{2\delta-1}}, &n \in [\delta]_0\cr
0,  &\text{otherwise}\cr
\end{cases},
$$ where $a\coloneqq \max\bigg\{ 4, \dfrac{\delta - 1}{2}\bigg\}$. Then for all $\ell\in [2\delta - 1]_0$, $\widecheck{\mathbf{m}}_{\ell}^{(p,m)} = \overline{S_{\ell}\widetilde{\mathbf{p}} \circ \mathbf{m}}$ satisfies 
\begin{align}\label{eqn: mask modulation}
\widecheck{\mathbf{m}}_{\ell}^{(p,m)} = e^{\tfrac{2\pi \mathbbm{i} \ell^2}{2\delta-1}} \cdot \widecheck{\mathbf{m}}_{2\ell  \; \text{mod} \; 2\delta-1}^{(\text{fpr})},
\end{align}
where $\widecheck{\mathbf{m}}_{\ell}^{(\text{fpr})}$ is defined as in \eqref{equ:Goodmasks}. 
As a consequence, if we let $\widecheck{\mathbf{M}}^{(\text{fpr})}$ and $\widecheck{\mathbf{M}}^{(p,m)}$ be the lifted linear measurement matrices as per \eqref{eqn: block-circulant matrix M} obtained by setting each $\widecheck{\mathbf{m}}_\ell$  in \eqref{eqn: Mcheck_k} equal to $\widecheck{\mathbf{m}}^{(\text{fpr})}_\ell$ and $\widecheck{\mathbf{m}}_{\ell}^{(p,m)}$, respectively, then we will have 
\begin{equation}\label{eqn: circulant same}
    \widecheck{\mathbf{M}}^{(p,m)}={\mathbf{P}}\widecheck{\mathbf{M}}^{(\text{fpr})},
\end{equation}
where $\mathbf{P}$ is a $D \times D$ block diagonal permutation matrix. Thus $\widecheck{\mathbf{M}}^{({p,m})}$ and $\widecheck{\mathbf{M}}^{({fpr})}$ have the same singular values and $$\kappa\left(\widecheck{\mathbf{M}}^{(p,m)}\right) =\kappa\left(\widecheck{\mathbf{M}}^{(\text{fpr})}\right)\leq C\delta^2,$$
where $\kappa(\cdot)$ denotes the condition number of a matrix.
\end{lemma}

\begin{proof} Using the definition of the Hadamard product $\circ$, the circulant shift operator $S_\ell$ and the reversal operator $\mathbf{x}\mapsto\widetilde{\mathbf{x}}$, we see that
\begin{align*}
(\widecheck{\mathbf{m}}_{\ell}^{(p,m)})_n = \overline{(S_\ell \widetilde{\mathbf{p}} \circ \mathbf{m})_{n}} = \overline{(S_\ell \widetilde{\mathbf{p}})_n \mathbf{m}_n} = \widetilde{\overline{\mathbf{p}}}_{n+\ell} \overline{\mathbf{m}_n}= \overline{\mathbf{p}}_{-n-\ell} \overline{\mathbf{m}_n}.
\end{align*}
Therefore, inserting the definitions of $\mathbf{p} $ and $\mathbf{m}$ above shows that for $n\in[\delta]_0$  
\begin{align*}\nonumber
(\widecheck{\mathbf{m}}_{\ell}^{(p,m)})_n &= e^{\tfrac{2 \pi \mathbbm{i} (n+\ell)^2}{2\delta-1}} \cdot \frac{e^{-(n+1)/a}}{\sqrt[\leftroot{-2}\uproot{2}4]{2\delta - 1}}\cdot e^{-\tfrac{2\pi \mathbbm{i} n^2}{2\delta-1}}&\\
&= e^{\tfrac{2 \pi \mathbbm{i} n^2}{2\delta-1}} \cdot e^{\tfrac{4 \pi \mathbbm{i} n\ell}{2\delta-1}} \cdot e^{\tfrac{2 \pi \mathbbm{i} \ell^2}{2\delta-1}} \cdot \frac{e^{-(n+1)/a}}{\sqrt[\leftroot{-2}\uproot{2}4]{2\delta - 1}}\cdot e^{-\tfrac{2 \pi \mathbbm{i} n^2}{2\delta-1}}&\\\nonumber
&= e^{\tfrac{2 \pi \mathbbm{i} \ell^2}{2\delta-1}}\bigg(\frac{e^{-(n+1)/a}}{\sqrt[\leftroot{-2}\uproot{2}4]{2\delta - 1}} \cdot e^{\frac{2\pi \mathbbm{i} n (2\ell)}{2\delta - 1}}\bigg) ~=~ e^{\tfrac{2 \pi \mathbbm{i} \ell^2}{2\delta-1}} \left( \widecheck{\mathbf{m}}_{2 \ell  \; \text{mod} \; 2\delta-1}^{(\text{fpr})} \right)_n.&\nonumber
\end{align*}
For $n\notin[\delta]_0$, we have  $\left(\widecheck{\mathbf{m}}_{\ell}^{(p,m)} \right)_n=e^{\tfrac{2 \pi \mathbbm{i} \ell^2}{2\delta-1}} \left( \widecheck{\mathbf{m}}_{2 \ell  \; \text{mod} \; 2\delta-1}^{(\text{fpr})} \right)_n=0$.  
Thus  \eqref{eqn: mask modulation} follows.

To prove \eqref{eqn: circulant same}, let $\widecheck{\mathbf{M}}^{(p,m)}$ and $\widecheck{\mathbf{M}}^{(p,m)}_k$ be the matrices obtained by  using the mask $\widecheck{\mathbf{m}}_\ell^{p,m}$ in \eqref{eqn: block-circulant matrix M} and \eqref{eqn: Mcheck_k} and let $\widecheck{\mathbf{M}}^{(\text{fpr})}$ and $\widecheck{\mathbf{M}}^{(\text{fpr})}_k$ be the matrices obtained using $\mathbf{m}^{(\text{fpr})}_\ell$ instead.
Then combining \eqref{eqn: mask modulation} and \eqref{eqn: Mcheck_k} implies
that $\left(\widecheck{\mathbf{M}}^{(p,m)}_k\right)_{i,j} =
\left(\widecheck{\mathbf{M}}^{(\text{fpr})}_k \right)_{2i \; \text{mod} \; 2\delta-1, j}. $ For example, when $ j \in [\delta - k + 1]_0$ one may check
\begin{align} \label{equ:Permdef}
\left(\widecheck{\mathbf{M}}^{(p,m)}_k\right)_{i,j} ~&=~ e^{\tfrac{2 \pi {\mathbbm{i}} i^2}{2\delta-1}} \left(\widecheck{\mathbf{m}}_{2 i  \; \text{mod} \; 2\delta-1} ^{(\text{fpr})}\right)_{k} \overline{e^{\tfrac{2\pi \mathbbm{i} i^2}{2\delta-1}}\left(\widecheck{\mathbf{m}}_{2i  \; \text{mod} \; 2\delta-1}^{(\text{fpr})} \right)_{j + k}}\\ ~&=~ 
\left(\widecheck{\mathbf{M}}^{(\text{fpr})}_k \right)_{2i \; \text{mod} \; 2\delta-1, j}, \nonumber
\end{align}
and one may perform similar computations in the other cases. Since each $\widecheck{\mathbf{M}}_k^{(p,m)}$ and $\widecheck{\mathbf{M}}_k^{(\text{fpr})}$ have $2\delta-1$ rows and the mapping $i\rightarrow 2i$ is a bijection on $\mathbb{Z}/(2\delta-1)\mathbb{Z}$ we see that each  $\widecheck{\mathbf{M}}_k^{(p,m)}$ may be obtained by permuting the rows of $\widecheck{\mathbf{M}}_k^{(\text{fpr})}$ (and that the permutation does not depend on $k$). Therefore, there exists a block diagonal permutation matrix $\mathbf{P}$ such that $\widecheck{\mathbf{M}}^{(p,m)}={\mathbf{P}}\widecheck{\mathbf{M}}^{(\text{fpr})}$. Finally, the  condition number bound for $\widecheck{\mathbf{M}}^{(p,m)}$ now follows from Theorem \ref{Condition Theorem} and the fact that permuting the rows of a matrix does not change its condition number or any of its singular values.
\end{proof}

Lemma~\ref{lem: rewrite} above demonstrates how to recast NFP problems involving locally supported masks and periodic PSFs as particular types of FFP problems.  Then, Lemma~\ref{lem: Choice of p and m} provides a particular PSF and mask combination for which the resulting FFP problem can be solved by inverting a well-conditioned linear system. 
Together they imply that, for properly chosen $\mathbf{m}$ and $\mathbf{p}$, one may robustly invert the measurements given in \eqref{eqn: noisy near field} by first recasting the NFP data as modified FFP data and then using the BlockPR approach from \cite{IVW16,IPSV18,preskitt2021admissible}.
This is the main idea behind Algorithm~\ref{NFP-BlockPR}. However, this approach will lead to theoretical error bounds which scale \emph{quadratically} in $d$. To remedy this, the final step of Algorithm~\ref{NFP-BlockPR} uses an alternative angular synchronization method (which originally appeared in \cite{preskitt2018phase}) based on a weighted graph Laplacian as opposed to previous works which used methods based on, e.g., the methods outlined in Section~\ref{sec: Eigenvector-Based Angular Synchronization}.  As we shall see in the next section, this will allow us to obtain bounds in Theorem~\ref{thm:MainThm}
which depend linearly on $d$ rather than quadratically.

\begin{algorithm}[H]
\caption{NFP-BlockPR} \label{NFP-BlockPR}
\begin{algorithmic}
\Require \\ 1) Variables $d, \delta, D = d(2\delta-1)$.\\
2) A $2\delta - 1$ periodic PSF $\mathbf{p} \in \mathbb{C}^d$, and a mask $\mathbf{m} \in \mathbb{C}^d$ with ${\rm supp}(\mathbf{m}) \subseteq [\delta]_0$.\\
3) A near-field ptychographic measurement matrix $\mathbf{Y} \in \mathbb{C}^{d \times 2\delta-1}$.
\Ensure $\mathbf{x}_{\text{est}}$ with $\mathbf{x}_{\text{est}} \approx e^{i \theta}\mathbf{x}$ for some $\theta \in [0,2\pi]$.
\State 1) Form masks $\widecheck{\mathbf{m}}_{\ell}^{(p,m)} = \overline{S_{\ell}\widetilde{\mathbf{p}} \circ \mathbf{m}}$ and matrix $\widecheck{\mathbf{M}}^{(p,m)}$ as per \eqref{eqn: block-circulant matrix M} and \eqref{eqn: Mcheck_k}.
\State 2)  Compute $\mathbf{z} = \left( \widecheck{\mathbf{M}}^{(p,m)} \right)^{-1} \text{vec}(\mathbf{Y}) \in \mathbb{C}^{D}$. 
\State 3) Reshape $\mathbf{z}$ to get $\widehat{\mathbf{X}}$ as per Section~\ref{sec: Eigenvector-Based Angular Synchronization} containing estimated entries of $\mathbf{x}\mathbf{x}^*$.
\State 4) Use weighted angular synchronization (Algorithm 3, \cite{preskitt2018phase}) to obtain $\mathbf{x}_{\text{est}}$.
\end{algorithmic}
\end{algorithm}
%
%
\section{Error Analysis for Algorithm~\ref{NFP-BlockPR}} \label{sec:ImprovedBlockPRanalysis}

In this section, we will prove our main result, Theorem~\ref{thm:MainThm}, which provides accuracy and robustness guarantees for Algorithm~\ref{NFP-BlockPR}.
For $\mathbf{x}\in\mathbb{C}^d$, we write its $n^{\text{th}}$ entry as $x_n\eqqcolon\vert x_n\vert e^{\mathbbm{i}\theta_n}$ and let $\mathbf{x}^{(\text{mag})} \coloneqq (\vert x_0\vert,\ldots,\vert x_{d-1}\vert)^T$ and $\mathbf{x}^{(\theta)} \coloneqq (e^{\mathbbm{i}\theta_0}, e^{\mathbbm{i}\theta_1}, \ldots, e^{\mathbbm{i}\theta_{d-1}})^T$ so that we may decompose $\mathbf{x}$ as 
\begin{equation}
\label{eqn: x decomp}
    \mathbf{x}=\mathbf{x}^{(\text{mag})} \circ \mathbf{x}^{(\theta)}.
\end{equation}
The following lemma upper bounds the total estimation error in terms of its phase and magnitude errors. For a proof, please see Appendix \ref{secA}. 

\begin{lemma}\label{lem: Full Error Bound} Let $\mathbf{x}$ be decomposed as in \eqref{eqn: x decomp}, and similarly let $\mathbf{x}_{\text{est}}$ be decomposed $\mathbf{x}_{\text{est}} = \mathbf{x}^{(\text{mag})}_{\text{est}} \circ \mathbf{x}^{(\theta)}_{\text{est}}$. Then, we have that
\begin{align} \label{lem:decoupleError}
\hspace{-1mm}\min_{\phi \in [0,2\pi)} \hspace{-.05in} \left\|\mathbf{x} - e^{\mathbbm{i}\phi}\mathbf{x}_{\text{est}}\right\|_{2} &\leq \|\mathbf{x}\|_{\infty} \hspace{-.05in} \min_{\phi \in [0,2\pi)} \hspace{-.03in} \left\|\mathbf{x}^{(\theta)}_{\text{est}} - e^{\mathbbm{i}\phi}\mathbf{x}^{(\theta)} \right\|_{2} \hspace{-.01in}+ \hspace{-.01in} \left\| \mathbf{x}^{(\text{mag})} -  \mathbf{x}^{(\text{mag})}_{\text{est}} \right\|_{2}. 
\end{align}

\end{lemma}
 In light of Lemma \ref{lem: Full Error Bound}, to bound the total error of our algorithm, it suffices to consider the phase and magnitude errors separately.
In order to bound $\| \mathbf{x}^{(\text{mag})} -  \mathbf{x}^{(\text{mag})}_{\text{est}}\|_{2}$, we may utilize the following lemma which is a restatement of Lemma 3 of  \cite{IVW16}.  

\begin{lemma}[Lemma 3 of \cite{IVW16}] \label{Mag Diff Error} Let $\sigma_{\min}\left(\widecheck{\mathbf{M}}^{(p,m)}\right)$ denote the smallest singular value of the lifted measurement matrix $\widecheck{\mathbf{M}}^{(p,m)}$ from line 1 of Algorithm~\ref{NFP-BlockPR}. Then,
\begin{align*}
\left\| \mathbf{x}^{(\text{mag})} -  \mathbf{x}^{(\text{mag})}_{\text{est}}\right\|_{\infty} \leq C\sqrt{ \frac{\|\mathbf{N}\|_{F}}{\sigma_{\rm min}\left(\widecheck{\mathbf{M}}^{(p,m)}\right)}}.
\end{align*}
\end{lemma}

Having obtained Lemma~\ref{Mag Diff Error}, we are now able to prove the following theorem bounding the total estimation error.

\begin{theorem}\label{thm: Particular Full Error Bound} 
Let $\mathbf{p}$ and $\mathbf{m}$ be the admissible PSF, mask pair defined in Lemma \ref{lem: Choice of p and m}.  Then, we have that
\begin{align*}
\min_{\phi \in [0,2\pi)} \left\|\mathbf{x} - e^{\mathbbm{i}\phi}\mathbf{x}_{\text{est}}\right\|_{2} \leq \|\mathbf{x}\|_{\infty} \min_{\phi \in [0,2\pi)} \left\|\mathbf{x}^{(\theta)}_{\text{est}} - e^{\mathbbm{i}\phi}\mathbf{x}^{(\theta)}\right\|_{2} + C\sqrt{d\delta \|\mathbf{N}\|_{F}}.
\end{align*}
\end{theorem}

\begin{proof}
Combining Lemmas \ref{lem: Full Error Bound}~and~\ref{Mag Diff Error} along with the inequality $\|\mathbf{u}\|_2\leq \sqrt{d}\|\mathbf{u}\|_\infty$, implies that 
\begin{align}\label{eqn: general t3}
    \min_{\phi \in [0,2\pi)} \hspace{-.05in} \|\mathbf{x} - e^{\mathbbm{i}\phi}\mathbf{x}_{\text{est}}\|_{2} &\leq \|\mathbf{x}\|_{\infty} \hspace{-.05in} \min_{\phi \in [0,2\pi)} \hspace{-.03in} \left\|\mathbf{x}^{(\theta)}_{\text{est}} - e^{\mathbbm{i}\phi}\mathbf{x}^{(\theta)} \right\|_{2} \hspace{-.01in}+ \hspace{-.01in} C\sqrt{ \frac{d \|\mathbf{N}\|_{F}}{\sigma_{\rm min}\left(\widecheck{\mathbf{M}}^{(p,m)}\right)}}.
\end{align}
As noted in Lemma~\ref{lem: Choice of p and m}, the singular values of $\widecheck{\mathbf{M}}^{(p,m)}$ are the same as those of $\widecheck{\mathbf{M}}^{(\text{fpr})}$. Therefore, applying Theorem~\ref{Condition Theorem} then finishes the proof.
\end{proof}

\begin{remark}
Note that the inequality \eqref{eqn: general t3} in the proof of Theorem~\ref{thm: Particular Full Error Bound} holds any time ${\bf m} \in \mathbb{C}^d$ satisfies ${\rm supp}({\bf m}) \subseteq [\delta]_0$ and either $(i)$ ${\bf p} \in \mathbb{C}^d$ is $2 \delta - 1$ periodic for $2 \delta - 1$ dividing $d$, or else $(ii)$ one instead collects NFP measurements \eqref{eqn: noisy near field} at all $(k,\ell) \in \mathcal{S}'$ as in Remark~\ref{rmk: full rewritten}. Therefore, results analogous to Theorem \ref{thm: Particular Full Error Bound} may be produced for any such $\mathbf{p}$ and $\mathbf{m}$ pairs such that $\sigma_{\rm min}\left(\widecheck{\mathbf{M}}^{(p,m)}\right) > 0$.  Furthermore, the value of this minimal singular value is straightforward to check numerically in practice.
\end{remark}

 In order to bound $\left\|\mathbf{x}^{(\theta)}_{\text{est}} - e^{\mathbbm{i}\phi}\mathbf{x}^{(\theta)} \right\|_{2}$, we will need a few additional definitions. 
 As in \eqref{eqn: bigX}, let $\mathbf{X}$ denote the partial autocorrelation matrix corresponding to the true signal $\mathbf{x}$, as in \eqref{eqn: bigXhat}, and let  $\widehat{\mathbf{X}}$ denote the partial autocorrelation matrix corresponding to $\mathbf{x}_{\text{est}}$, i.e., the matrix obtained in  step 3 of Algorithm~\ref{NFP-BlockPR}.
Let $G = (V,E,\mathbf{W})$ be a weighted graph whose vertices are given by $V = [d]_0$, whose edge set $E$ is taken to be the set of  $(i,j)$ such that $i\neq j$ and $\lvert i-j\rvert  \text{ mod } d <\delta$, and whose weight matrix $\mathbf{W}$ is  
defined entrywise by 
\begin{align}\label{equ:weights}
W_{i,j} = 
\begin{cases}
\lvert\widehat{X}_{i,j}\rvert^2, \; &0 < \lvert i - j \rvert \; \text{mod} \; d < \delta\cr
0, \; &\text{otherwise}\cr
\end{cases}.
\end{align}
Letting $\mathbf{A}_G$ denote the \textit{unweighted} adjacency matrix of $G$, we observe that by construction, we have $\mathbf{X} = (\mathbf{I} + \mathbf{A}_G) \circ \mathbf{x}\mathbf{x}^*$ and  $\widehat{\mathbf{X}} = (\mathbf{I} + \mathbf{A}_G) \circ \mathbf{x}_{\text{est}}\mathbf{x}_{\text{est}}^*$. Letting $\mathbf{D}$ denote the \textit{weighted} degree matrix, we define the \textit{unnormalized graph Laplacian} by  $\mathbf{L}_G \coloneqq \mathbf{D} - \mathbf{W}$ and the \textit{normalized graph Laplacian} by $\mathbf{L}_N \coloneqq \mathbf{D}^{-1/2} \mathbf{L}_G \mathbf{D}^{-1/2}$. It is well known that both $\mathbf{L}_G$ and  $\mathbf{L}_N$ are positive semi-definite with a minimal eigenvalue of zero (see, e.g., Section 3.1, \cite{sagt}). We will let $\tau_G$ denote the spectral gap (second smallest eigenvalue) of $\mathbf{L}_G$. It is well known that if $G$ is connected then $\tau_G$ is strictly positive (see, e.g., Lemma 3.1.1, \cite{sagt}). 

In \cite{filbir2021recovery}, the authors used a weighted graph approach to prove the following result which bounds $\underset{\phi \in [0,2\pi)}{\min} \|\mathbf{x}^{(\theta)}_{\text{est}} - e^{\mathbbm{i}\phi}\mathbf{x}^{(\theta)}\|_{2}$ .

\begin{theorem}[Corollary 3, \cite{filbir2021recovery}]\label{thm: Improved Phase Bound} Consider the weighted graph $G = (V,E,\mathbf{W})$ described in the previous paragraph with weight matrix given as in \eqref{equ:weights}.
Let $\tau_{G}$ denote the spectral gap of the associated unnormalized Laplacian $\mathbf{L}_{G}$. Then we have that
\begin{align*}
 \underset{\phi \in [0,2\pi)}{\min} \left\|\mathbf{x}^{(\theta)}_{\text{est}} - e^{\mathbbm{i}\phi}\mathbf{x}^{(\theta)}\right\|_{2} \leq C \sqrt{1 + \|\mathbf{x}_{\rm{est}}\|_{\infty}} \cdot \dfrac{\|\mathbf{X} - \widehat{\mathbf{X}}\|_F}{\sqrt{\tau_{G}}}, \; C \in \mathbb{R}^+.
\end{align*}
\end{theorem}
\begin{remark} The $\sqrt{1 + \|\mathbf{x}_{\text{est}}\|_{\infty}}$ term is referred to in Theorem 4 of \cite{filbir2021recovery} as a \textit{tightness penalty} which is applied when taking the non-convex constraint and performing an eigenvector relaxation, allowing us to use the method of angular synchronization involving the weighted Laplacian given in Algorithm 3 of \cite{preskitt2018phase}.
\end{remark}
In order to utilize Theorem~\ref{thm: Improved Phase Bound} we require both an upper bound of $\|\mathbf{X} - \widehat{\mathbf{X}}\|_F$ and a lower bound for the spectral gap $\tau_G$.  These are provided by the next two lemmas.

\begin{lemma} \label{lem: Frobenius Bound} Let $\mathbf{p}$ and $\mathbf{m}$ be defined as in Lemma \ref{lem: Choice of p and m}.  Then, $\|\mathbf{X} - \widehat{\mathbf{X}}\|_{F} \leq C \delta \|\mathbf{N}\|_{F}$.
\end{lemma}

\begin{proof} 

Let ${\rm vec}: \mathbb{C}^{d \times d} \rightarrow \mathbbm{C}^D$ 
be the vectorization operator considered in \eqref{eqn: Y equation}. 
It follows from \eqref{equ:FFPmeasurements}, \eqref{eqn: Y equation}, and Step 2 of Algorithm \ref{NFP-BlockPR}, that
$$\text{vec}(\mathbf{Y}) = \widecheck{\mathbf{M}}\text{vec}(\widehat{\mathbf{X}})\quad\text{and}\quad\text{vec}(\mathbf{Y}-\mathbf{N}) = \widecheck{\mathbf{M}}\text{vec}(\mathbf{X}). $$
Therefore,
\begin{align*}
    \left\|\mathbf{X}-\widehat{\mathbf{X}}\right\|_F \leq \left\| \left(\widecheck{\mathbf{M}}^{(p,m)}\right)^{-1} \text{vec}(\mathbf{N}) \right\|_2
    \leq \frac{\|\text{vec}(\mathbf{N})\|_2}{\sigma_{\rm min}\left(\widecheck{\mathbf{M}}^{(p,m)}\right)} \leq 
    C \delta \|\mathbf{N}\|_F,
\end{align*}
where final inequality again utilizes  Lemma~\ref{lem: Choice of p and m} and Theorem~\ref{Condition Theorem}.
\end{proof}

\begin{lemma} \label{lem: spectral gap}
For the graph $G$ considered in Theorem~\ref{thm: Improved Phase Bound}, we have that $$\tau_{G} \geq \dfrac{\lvert\mathbf{x}_{\text{est}}\rvert_{\min}^4}{\|\mathbf{x}_{\text{est}}\|_{\infty}^2} \dfrac{4(\delta-1)}{d^2}.$$
\end{lemma}
\begin{proof}
Letting $W_{\min}$ and $W_{\text{max}}$ be the minimum and maximum value of any of the (nonzero) entries of $\mathbf{W}$, we have that $W_{\min}\geq \lvert \mathbf{x}_{\text{est}}\rvert^2_{\min}$,  $W_{\max}\leq \|\mathbf{x}_{\text{est}}\|^2_{\infty}$, and $\text{diam}(G_{\text{unw}})\geq d/(2\delta-1)$ (where $\text{diam}(G_{\text{unw}})$ is the diameter of the unweighted version of $G$). Therefore,
by Theorem \ref{General Weighted Spectral Gap} in Appendix \ref{secB}, we have that
\begin{align*}
\tau_{G} \geq \dfrac{\lvert\mathbf{x}_{\text{est}}\rvert_{\min}^4}{\|\mathbf{x}_{\text{est}}\|_{\infty}^2}\dfrac{2}{(d-1)\text{diam}(G)} \geq 
 \dfrac{\lvert\mathbf{x}_{\text{est}}\rvert_{\min}^4}{\|\mathbf{x}_{\text{est}}\|_{\infty}^2} \dfrac{4(\delta-1)}{d^2}.
\end{align*}
\end{proof}

We shall now finally prove our main result.
\begin{proof}[The Proof of Theorem \ref{thm:MainThm}]
By Theorem~\ref{thm: Particular Full Error Bound}, we have 
\begin{align*}
\min_{\phi \in [0,2\pi)} \left\|\mathbf{x} - e^{\mathbbm{i}\phi}\mathbf{x}_{\text{est}}\right\|_{2} \leq \|\mathbf{x}\|_{\infty} \min_{\phi \in [0,2\pi)} \left\|\mathbf{x}^{(\theta)}_{\text{est}} - e^{\mathbbm{i}\phi}\mathbf{x}^{(\theta)}\right\|_{2} + C\sqrt{d\delta \|\mathbf{N}\|_{F}}.
\end{align*} 
Combining Theorem ~\ref{thm: Improved Phase Bound} with Lemmas~\ref{lem: Frobenius Bound}~and~\ref{lem: spectral gap} yields
\begin{align*}
 \underset{\phi \in [0,2\pi)}{\min} \left\|\mathbf{x}^{(\theta)}_{\text{est}} - e^{\mathbbm{i}\theta}\mathbf{x}^{(\theta)}\right\|_{2} &\leq C \sqrt{1 + \|\mathbf{x}_{\rm{est}}\|_{\infty}} \cdot \dfrac{\|\mathbf{X} - \widehat{\mathbf{X}}\|_F}{\sqrt{\tau_{G}}}\\
&\leq C \sqrt{1 + \|\mathbf{x}_{\rm{est}}\|_{\infty}} \cdot \frac{d\sqrt{\delta}\|\mathbf{x}_{\text{est}}\|_{\infty}\|\mathbf{N} \|_F}
{\lvert\mathbf{x}_{\text{est}}\rvert_{\min}^2}.
\end{align*}The result follows.
\end{proof}



\section{An Alternate Approach:  Near-Field Ptychography via Wirtinger Flow}
\label{sec:WTF4NFP}


In the previous sections we have demonstrated a particular point spread function and mask for which NFP measurements are guaranteed to allow image reconstruction via Algorithm~\ref{NFP-BlockPR}. However, in many real-world scenarios the particular mask and PSF combination considered above are not of the type actually used in practice. For example, in the setting considered in \cite{zhang2019near} the PSF $\mathbf{p}$ ideally behaves like a low-pass filter (so that, e.g., $\widehat{\mathbf{p}}$ is supported in $\{ k\in\mathbb{Z} \vert -K< k \text{ mod } d < K\}$ for some $K \ll d$), and the mask $m$ is globally supported in $[d]_0$.  In contrast, the PSF considered above has its nonzero Discrete Fourier coefficients at frequencies in $\{k d / (2 \delta - 1)\}_{k \in  [2 \delta - 1]_0}$ (and thus its Fourier support includes large frequencies), and the mask $m$ has small physical support in $[\delta]_0$.  
This motivates us to explore a variant of the well known Wirtinger Flow algorithm \cite{candes2015phasewirtinger} in this section. This  method, Algorithm \ref{NFP Wirtinger Flow}, can be applied to more general set of PSF and mask pairs than Algorithm \ref{NFP-BlockPR} considered in the previous section.



Suppose we have noiseless NFP measurements of the form
\begin{align*}
Y_{k,\ell} = \lvert (\mathbf{p} * (S_k \mathbf{m} \circ \mathbf{x}))_\ell\rvert^2, ~~ (k,\ell) \in \hspace{-.2in} \bigcup_{0\leq k \leq K-1} \{d - k\} \times \{K-L+1,\dots,k\} \hspace{-.1in} \mod d,
\end{align*}
where $K,L \in [d+1]_0 \setminus \{ 0 \}$.
Then by the same argument used in Lemma \ref{lem: rewrite} (see also in Remark \ref{rmk: full rewritten}), we can manipulate the measurements above so that we have 
\begin{align*}
\widetilde{Y}_{k,\ell} = \lvert\langle \widecheck{\mathbf{m}}_{\ell}^{(p,m)} , S_k\mathbf{x} \rangle\rvert^2, \quad (k,\ell) \in [K]_0 \times [L]_0,
\end{align*}
where the masks $\widecheck{\mathbf{m}}_{\ell}^{(p,m)}$ are defined as in \eqref{equ:NFPmaskdef}. We may then reshape these measurements into a vector $\mathbf{y}\in\mathbb{C}^{KL}$ with entries given by
\begin{align} \label{equ:vecmeasdef}
y_n := \lvert\langle \widecheck{\mathbf{m}}^{(p,m)}_{n \hspace{-.05in} \mod L} , S_{\lfloor\frac{n}{L}\rfloor} \mathbf{x} \rangle\rvert^2, \quad \forall n \in [KL]_0.
\end{align}
After this reformulation, we may then apply a standard Wirtinger Flow Algorithm with spectral initialization. Full details are given below in Algorithm~\ref{NFP Wirtinger Flow}.

\begin{algorithm}[H]
\caption{NFP Wirtinger Flow} \label{NFP Wirtinger Flow}
\begin{algorithmic}

\Require \\
1) Size $d \in \mathbb{N}$, 
number of iterations $T$, stepsizes $\mu_{\tau +1}$ for $\tau \in [T]_0$.\footnotemark \\
2) PSF $\mathbf{p} \in \mathbb{C}^d$, mask $\mathbf{m} \in \mathbb{C}^d$, $\widecheck{\mathbf{m}}_{\ell}^{(p,m)} = \overline{S_{\ell} \widetilde{\mathbf{p}} \circ \mathbf{m}}$.\\
3) Noisy measurements $Y_{k,\ell} = \lvert (\mathbf{p} * (S_k \mathbf{m} \circ \mathbf{x}))_\ell \rvert^2 + N_{k,\ell}$.
\Ensure $\mathbf{x}_{\text{est}} \in \mathbb{C}^d$ with $\mathbf{x}_{\text{est}} \approx e^{i \theta}\mathbf{x}$ for some $\theta \in [0,2\pi]$
\State 1) Rearrange measurement matrix to form measurement vector ${\bf y}$ in \eqref{equ:vecmeasdef}.
\State 2) Compute ${\bf z}_0$ using spectral method (Algorithm 1 in \cite{candes2015phasewirtinger}).
\State 3) For $\tau \in [T]_0$, let $\mathbf{z}_{\tau + 1} = \mathbf{z}_\tau - \dfrac{\mu_{\tau +1}}{\|\mathbf{z}_0\|^2} \nabla f(\mathbf{z}_\tau)$ where\\ \begin{center}
$f(\mathbf{z}) \coloneqq \dfrac{1}{KL} \sum_{n = 1}^{KL} \left(\left\lvert \left(S_{-\lfloor\frac{n}{L}\rfloor}\widecheck{\mathbf{m}}^{(p,m)}_{n  ~{\rm mod}~ L} \right)^*\mathbf{z} \right\rvert^2  - y_n \right)^2$. \end{center}
\State 4) Return $\mathbf{x}_{\text{est}} = \mathbf{z}_{T}$.
\end{algorithmic}
\end{algorithm}
\footnotetext{For our numerical simulations in Section~\ref{sec: Numerics}, we set $\mu_\tau = \min(1 - e^{-\tau/330},0.4)$ as suggested in \cite{candes2015phasewirtinger}.}

\section{Numerical Simulations} \label{sec: Numerics}

In this section, we evaluate Algorithms~\ref{NFP-BlockPR} and~\ref{NFP Wirtinger Flow} with respect to both noise robustness and runtime.  Every data point in the plots below reports an average reconstruction error or runtime over 100 tests.  For each test, a new sample ${\bf x} \in \mathbb{C}^d$ is randomly generated by choosing each entry to have independent and identically distributed (i.i.d.) mean 0 and variance 1 Gaussian real and imaginary parts.  We then attempt to recover this sample  from the noisy measurements $Y_{k,\ell} \left( {\bf x} \right)$ defined as in \eqref{eqn: noisy near field} where the additive noise matrices $\mathbf{N}$ also have i.i.d. mean 0 Gaussian entries.

In our noise robustness experiments, we plot the reconstruction error as a function of the Signal-to-Noise  Ratio (SNR), where we define the reconstruction error by $${\rm Error}({\bf x}, \mathbf{x}_{\text{est}}) := 10\log_{10} \bigg( \frac{\min_{\phi}\|\mathbf{x} - e^{\mathbbm{i}\phi}\mathbf{x}_{\text{est}}\|_{2}^{2}}{\|\mathbf{x}\|_{2}^{2}}\bigg),$$
and the  SNR 
 by $${\rm SNR}(\mathbf{Y},\mathbf{N}) := 10\log_{10} \bigg(\frac{\|\mathbf{Y} - \mathbf{N}\|_{F}}{\|\mathbf{N}\|_{F}}\bigg).$$ In these experiments,  we re-scale the noise matrix $\mathbf{N}$ in order to achieve each desired SNR level.
All simulations were performed using MATLAB R2021b on an Intel desktop with a 2.60GHz i7-10750H CPU and 16GB DDR4 2933MHz memory. All code used to generate the figures below is publicly available at \url{https://github.com/MarkPhilipRoach/NearFieldPtychography}.

\subsection{Algorithms \ref{NFP-BlockPR} and \ref{NFP Wirtinger Flow} for Locally Supported Masks and Periodic Point Spread Functions}\label{sec:lem2Measurements}

In these experiments, we choose the  measurement index set for \eqref{eqn: noisy near field} to be $\mathcal{S} = \mathcal{K} \times \mathcal{L}$ where $\mathcal{K}=[d]_0$ and $\mathcal{L}=[2\delta-1]_0$.  As a consequence we see that we consider all shifts $k \in [d]_0$ of the mask while observing only a portion of each resulting noisy near-field diffraction pattern $\lvert \mathbf{p} * (S_k \mathbf{m} \circ \mathbf{x}) \rvert^2$ for each $k$.  This corresponds to a physical imaging system where, e.g., the sample and (a smaller) detector are fixed while a localized probe with support size $\delta$ scans across the sample.  Figure~\ref{Fig:Alg1lem2masks} evaluates the robustness and runtime of Algorithm~\ref{NFP-BlockPR} as a function of the SNR and mask support $\delta$ in this setting.  Looking at Figure~\ref{Fig:Alg1lem2masks} one can see that noise robustness increases with the support size of the mask, $\delta$, in exchange for mild increases in runtime.

\begin{figure}[H] 
\centering
\includegraphics[width=1\textwidth]{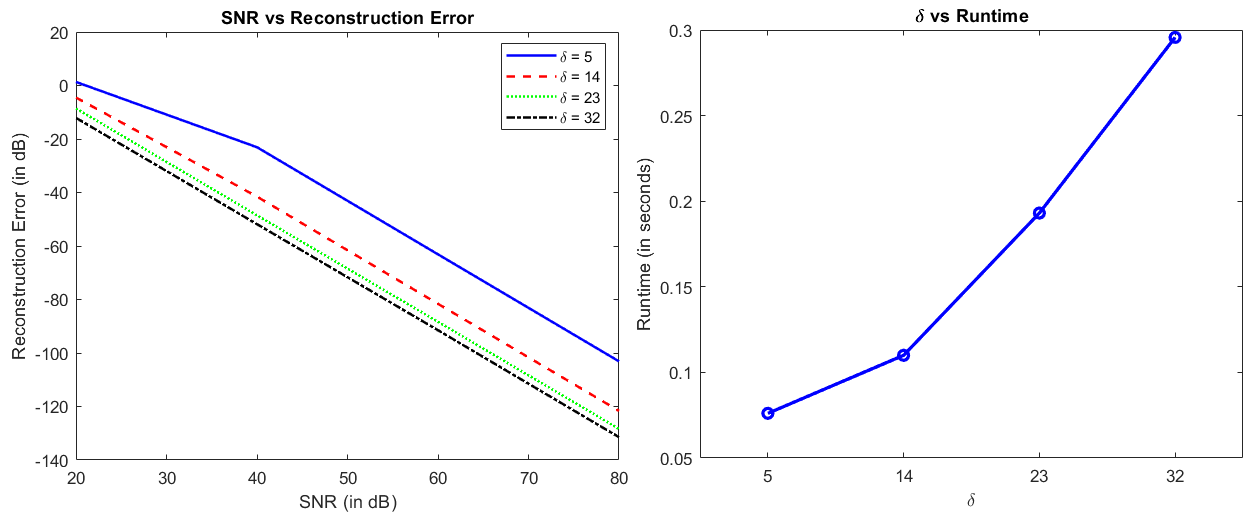}
\caption{An evaluation of Algorithm \ref{NFP-BlockPR} for the proposed PSF and mask with $d = 945$.  Left: Reconstruction error vs SNR for various $\delta = \lvert{\rm supp}({\bf m})\rvert$.  Right:  Runtime as a function of $\delta$.} \label{Fig:Alg1lem2masks}
\end{figure}

Figure~\ref{Fig:Algcomplem2mask} compares the performance of Algorithm \ref{NFP-BlockPR} and Algorithm \ref{NFP Wirtinger Flow} for the measurements proposed in Lemma~\ref{lem: Choice of p and m}.  Looking at Figure~\ref{Fig:Algcomplem2mask} we can see that Algorithm \ref{NFP Wirtinger Flow} takes longer to achieve comparable errors to Algorithm \ref{NFP-BlockPR} for these particular $\mathbf{p}$ and $\mathbf{m}$ as SNR increases. More specifically, we see, e.g., that BlockPR achieves a similar reconstruction error to 500 iterations of Wirtinger flow at an SNR of about $50$ in a small fraction of the time.  This supports the value of the BlockPR method as a fast initializer for more traditional optimization-based solution approaches.

\begin{figure}[H] 
\centering
\includegraphics[width=1\textwidth]{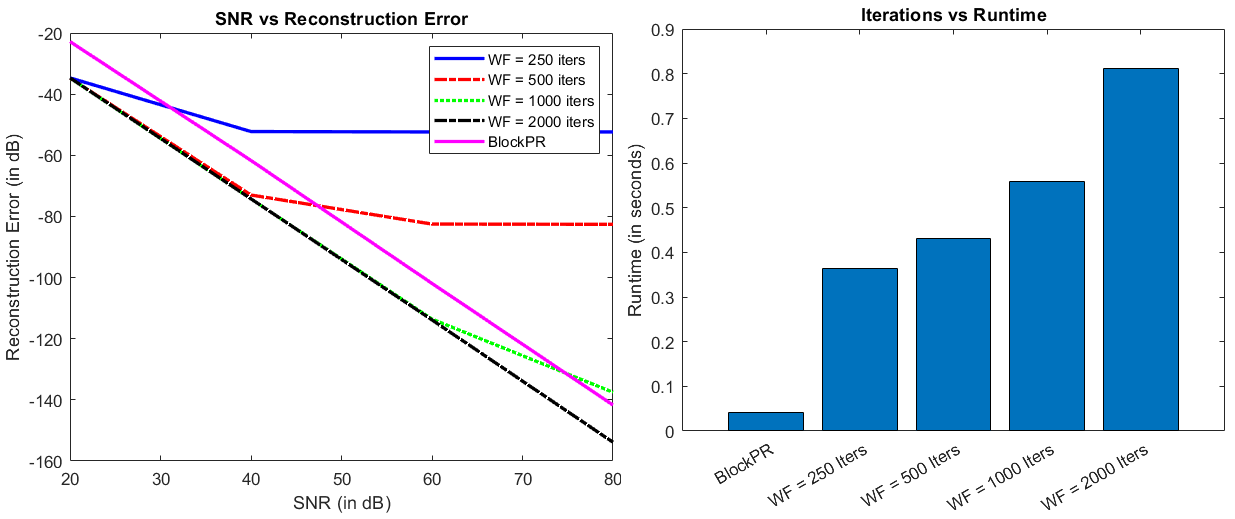}
\caption{A comparison of Algorithms \ref{NFP-BlockPR}~and~\ref{NFP Wirtinger Flow} for the proposed PSF and mask with $\delta = 26$ and $d = 102$.  Left: Reconstruction error vs SNR for various numbers of Algorithm~\ref{NFP Wirtinger Flow} iterations.  Right: The corresponding average runtimes.} 
\label{Fig:Algcomplem2mask}
\end{figure}

\subsection{Algorithm \ref{NFP Wirtinger Flow} for Globally Supported Masks}
\label{sec:GZexperiments}

As we saw in the previous section, Algorithm \ref{NFP-BlockPR} is able to invert NFP measurements more efficiently than Algorithm \ref{NFP Wirtinger Flow} in situations where it is applicable. However, Algorithm \ref{NFP-BlockPR} only applies to locally supported masks. In this section, we will show that Algorithm \ref{NFP Wirtinger Flow} remains effective even when the masks are globally supported, such as the masks considered in \cite{zhang2019near}.

In Figure~\ref{fig:Alg2wlessshifts}, we evaluate Algorithm~\ref{NFP Wirtinger Flow} using noisy measurements of the form
\begin{align} \label{equ:MeasAlg2experiments}
Y_{k,\ell} = \lvert (\mathbf{p} * (S_k \mathbf{m} \circ \mathbf{x}))_\ell\rvert^2 + N_{k,l}, ~~ (k,\ell) \in [K]_0 \times [d]_0.
\end{align}
Here $\mathbf{p} \in \mathbb{C}^d$ is a low-pass filter with $\widehat{\mathbf{p}} = S_{-(\gamma-1)/2}\mathbbm{1}_{\gamma}$ where $\gamma = \tfrac{d}{3}+1$ and $\mathbbm{1}_{\gamma} \in \{ 0,1 \}^d$ is a vector whose first $\gamma$ entries are $1$ and whose last $d-\gamma$ entries are $0$.  Here, we choose the  mask $\mathbf{m}$ to have i.i.d. mean 0 variance 1 Gaussian entries.  Thus, the measurements considered in  \eqref{equ:MeasAlg2experiments} differ from those used in Section~\ref{sec:lem2Measurements} in two crucial respects: i) the mask ${\bf m}$ here has global support.  ii) we utilize a small number of mask shifts and observe the entire diffraction pattern resulting from each one (as opposed to observing just a portion of each diffraction pattern from all possible shifts, as above).  Examining Figure~\ref{fig:Alg2wlessshifts}, one can see Algorithm \ref{NFP Wirtinger Flow} remains effective in this setting. We also observe, as expected, that using more shifts, i.e., collecting more measurements, results in lower reconstruction errors.

\begin{figure}[H] 
\centering
\includegraphics[width=1\textwidth]{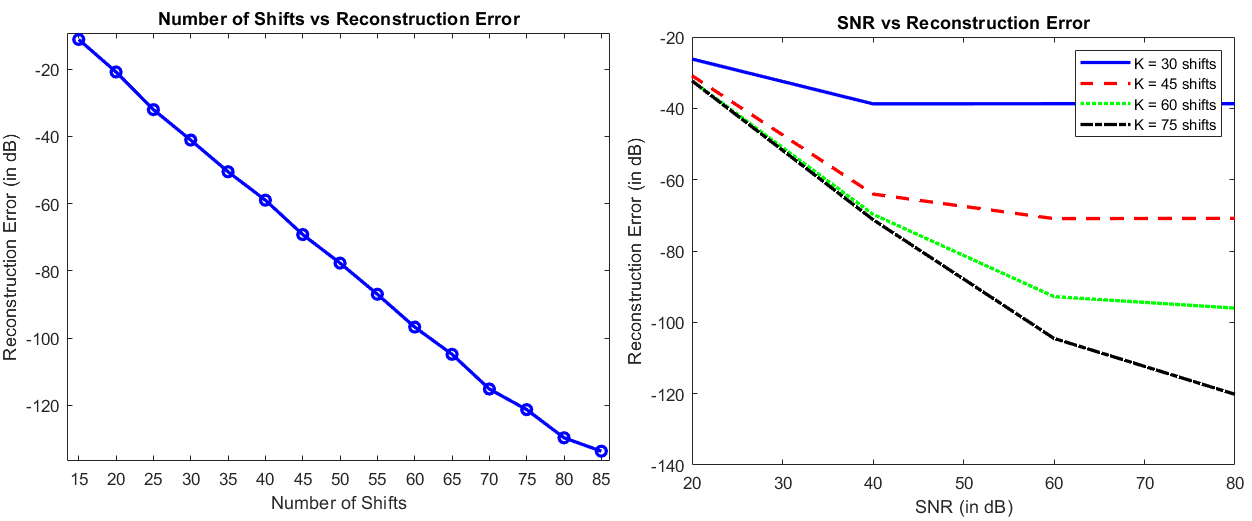}
\caption{The reconstruction error of Algorithm~\ref{NFP Wirtinger Flow} with $d = 102$, $\mathcal{L} = [d]_0$, and number of iterations $T = 2000$.  Left:  Reconstruction error vs the number of total shifts $K$ for fixed ${\rm SNR} = 80$.  Right:  Reconstruction error vs SNR for various numbers of shifts $K$.}  \label{fig:Alg2wlessshifts} 
\end{figure}

\subsection{Algorithm~\ref{NFP-BlockPR} for Non-Periodic PSFs via Remark~\ref{rmk: full rewritten}}

In these experiments, we choose the measurement index set for \eqref{eqn: noisy near field} to be $\mathcal{S}'$ from Remark~\ref{rmk: full rewritten} and consider two different non-periodic PSFs together with locally supported masks ${\bf m} \in \mathbb{C}^d$ having ${\rm supp}(\mathbf{m}) \subseteq [\delta]_0$ for $\delta = 26$ and $d = 102$.  Motivated again by \cite{zhang2019near}, we first consider a PSF given by a low-pass filter (defined as in Section \ref{sec:GZexperiments}) plus small noise modeling imperfections (here the additive vector has i.i.d. $\mathcal{N}(0,10^{-4})$ normal entries) in combination with a random symmetric mask.  Here the mask's nonzero entries are created by reflecting $\delta/2 = 13$ random entries (chosen via i.i.d. mean 0 variance 1 Gaussians) across the middle of its support.  The reconstruction error of Algorithm~\ref{NFP-BlockPR}, as well as of Algorithm~\ref{NFP Wirtinger Flow} initialized with the output of Algorithm~\ref{NFP-BlockPR}, is plotted on the left in Figure~\ref{Fig:Algremark1} as a function of the NFP measurements' SNR for this PSF/mask pair.  

For our second non-periodic PSF and locally supported mask pair, we let the PSF be a vector with unit magnitude entries having i.i.d. uniformly random phases, and let our locally supported masks have $\delta$ nonzero i.i.d. mean $0$ variance $1$ Gaussian entries.  The reconstruction errors of both Algorithm~\ref{NFP-BlockPR} and Algorithm~\ref{NFP Wirtinger Flow} are plotted on the right in Figure~\ref{Fig:Algremark1} as a function of the NFP measurements' SNR in this case.  In both experiments plotted in Figure~\ref{Fig:Algremark1}, we note that both random initialization as well as the spectral initialization method from \cite{candes2015phasewirtinger} appear insufficiently accurate to allow Algorithm~\ref{NFP Wirtinger Flow} to converge.  However, when the output of  Algorithm~\ref{NFP-BlockPR} is used to compute $z_0$ in step 2 of Algorithm~\ref{NFP Wirtinger Flow}, Algorithm~\ref{NFP Wirtinger Flow} then converges nicely to an accurate estimate of the true signal.  This further reinforces the potential value of Algorithm~\ref{NFP-BlockPR} as fast and accurate initializer for more traditional optimization-based solution approaches.

\begin{figure}[H] 
\centering
\includegraphics[width=1\textwidth]{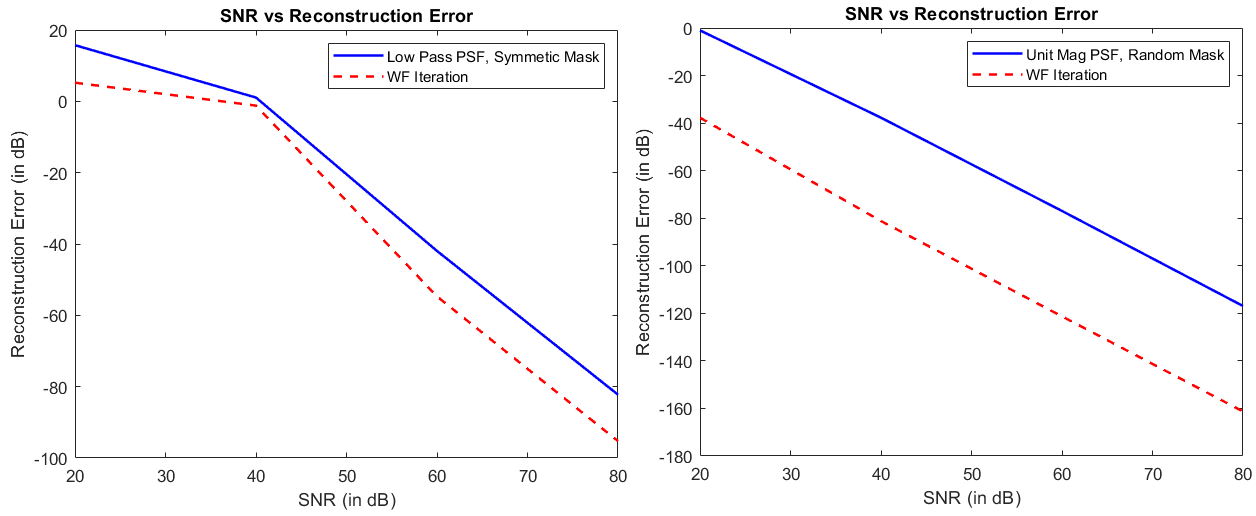}
\caption{A simulation applying Algorithm~\ref{NFP-BlockPR} via Remark~\ref{rmk: full rewritten} and then using its generated estimate as the initial estimate $z_0$ in Algorithm~\ref{NFP Wirtinger Flow}. In both plots Algorithm~\ref{NFP-BlockPR} is plotted in solid blue, and Algorithm~\ref{NFP Wirtinger Flow} is plotted in dashed red. Left: Reconstruction error vs SNR where the PSF is an low-pass filter with small additive noise, and the locally supported mask is symmetric and random. 
Here $T = 5000$ iterations are used in Algorithm~\ref{NFP Wirtinger Flow}.  Right: Reconstruction error vs SNR where the PSF is an vector with randomized unit magnitude entries, and the locally supported mask is random.  
Here $T = 1000$ iterations are used in Algorithm~\ref{NFP Wirtinger Flow}.} 
\label{Fig:Algremark1}
\end{figure}

\section{Conclusions and Future Work}
\label{sec:Conc}

We have introduced two new algorithms for recovering a specimen of interest from near-field ptychographic measurements. Both of these algorithms rely on first reformulating and reshaping our measurements so that they resemble widely-studied far-field ptychographic measurements. We then recover our method using either Wirtinger Flow or using methods based on \cite{IVW16}. Algorithm \ref{NFP-BlockPR} is computational efficient and, to the best of our knowledge, is the first algorithm with provable recovery guarantees for measurements of this form. Algorithm \ref{NFP Wirtinger Flow}, on the other hand, has the advantage of being applied to more general masks with global support. Developing more efficient and provably accurate algorithms for this latter class of measurements remains an interesting avenue for future work.


\section*{Acknowledgements}
The authors would like to thank Guoan Zheng for answering questions about (and providing code for) his work in \cite{zhang2019near}.  M. Roach's and M. Iwen's work on this project with both supported in part by NSF DMS 1912706.

\section*{Statements and Declarations}
On behalf of all authors, the corresponding author states that there is no conflict of interest.

\appendix
\section{Technical Lemmas}\label{secA}

We state the following lemmas for the sake of completeness.  Note that we index all vectors modulo $d$.
\begin{lemma}\label{Conjugation Identity} Let $\mathbf{x}, \mathbf{y} \in \mathbb{C}^{d}$. We have that $\lvert\langle \mathbf{x}, \overline{\mathbf{y}} \rangle\rvert^2 = \lvert\langle \overline{\mathbf{x}}, \mathbf{y} \rangle\rvert^2$.
\end{lemma}

\begin{lemma} \label{Lemma 2} Let $\mathbf{x}, \mathbf{y} \in \mathbb{C}^{d}$. We have that $\langle \mathbf{x}, \mathbf{y} \circ \mathbf{z} \rangle = \langle \mathbf{x} \circ \overline{\mathbf{y}}, \mathbf{z} \rangle$.
\end{lemma}
\begin{proof} By the definition of the inner product and the Hadamard product
\begin{align*}
\langle \mathbf{x}, \mathbf{y} \circ \mathbf{z} \rangle &=  \sum_{n=0}^{d-1} x_n \overline{(\mathbf{y} \circ \mathbf{z})_n} =  \sum_{n=0}^{d-1} x_n \overline{y}_n \overline{z}_n  = \sum_{n=0}^{d-1} (\mathbf{x} \circ \overline{\mathbf{y}})_n \overline{z}_n = \langle \mathbf{x} \circ \overline{\mathbf{y}}, \mathbf{z} \rangle.
\end{align*}
\end{proof}
\begin{lemma} \label{Shift Identities} Let $\mathbf{x}, \mathbf{y} \in \mathbb{C}^{d}, k \in \mathbb{Z}$. We have that 
 \begin{enumerate}
\item \label{Shift Lemma 1} $(\mathbf{x} * \mathbf{y})_k = \langle S_{-k} \widetilde{\mathbf{x}}, \overline{\mathbf{y}} \rangle$;
\item \label{Shift Lemma 2} $\mathbf{x} \circ S_k\mathbf{y} = S_k(S_{-k}\mathbf{x} \circ \mathbf{y})$;
\item \label{Shift Lemma 3} $\langle S_k \mathbf{x}, \mathbf{y} \rangle = \langle \mathbf{x}, S_{-k} \mathbf{y} \rangle$.
\end{enumerate}
\end{lemma}
\begin{proof} Proof of \ref{Shift Lemma 1}: Let $\mathbf{x}, \mathbf{y} \in \mathbb{C}^{d}$. By the definition of the circular convolution
\begin{align*}
(\mathbf{x} * \mathbf{y})_k &= \sum_{n=0}^{d-1} x_{k-n}y_n = \sum_{n=0}^{d-1} x_{k-n} \overline{\overline{y_n}} = \sum_{n=0}^{d-1} (S_{-k}\widetilde{\mathbf{x}})_n \overline{\overline{y_n}} = \langle S_{-k} \widetilde{\mathbf{x}}, \overline{\mathbf{y}} \rangle.&
\end{align*}
Proof of \ref{Shift Lemma 2}: Let $\mathbf{x}, \mathbf{y} \in \mathbb{C}^d, k \in \mathbb{Z}$. Let $n \in [d]_0$ be arbitrary.  Then we have that
\begin{align*}
(\mathbf{x} \circ S_k\mathbf{y})_n = \mathbf{x}_n (S_k\mathbf{y})_n = (S_{-k}\mathbf{x})_{n+k} \mathbf{y}_{n+k}  = (S_k (S_{-k}\mathbf{x} \circ \mathbf{y}))_n.
\end{align*}
Proof of \ref{Shift Lemma 3}: Noting that we index modulo $d,$ we have that
\begin{align*}
\langle S_k \mathbf{x}, \mathbf{y} \rangle = \sum_{n=0}^{d-1} (S_k \mathbf{x})_n \overline{y_n} = \sum_{n=0}^{d-1} x_{n+k} \overline{y_n} = \sum_{n=k}^{d+k-1} x_{n} \overline{y_{n-k}} = \sum_{n=0}^{d-1} \mathbf{x}_{n} \overline{y_{n-k}}=\langle\mathbf{x},S_{-k}\mathbf{y}\rangle.
\end{align*}
\end{proof}
We now give our proof of Lemma \ref{lem: Full Error Bound}.
\begin{proof}[Proof of Lemma \ref{lem: Full Error Bound}] Fix $\phi \in [0,2\pi)$.  By the triangle inequality we have

\begin{align}\nonumber
 \|\mathbf{x} - e^{\mathbbm{i}\phi}\mathbf{x}_{\text{est}}\|_{2} =& \| \mathbf{x}^{(\text{mag})} \circ \mathbf{x}^{(\theta)} - \mathbf{x}_{\text{est}}^{(\text{mag})} \circ e^{i \phi} \mathbf{x}_{\text{est}}^{(\theta)}\|_{2}&\\\nonumber
\leq&~ \| \mathbf{x}^{(\text{mag})} \circ \mathbf{x}^{(\theta)} -  \mathbf{x}^{(\text{mag})} \circ e^{\mathbbm{i}\phi}\mathbf{x}_{\text{est}}^{(\theta)}\|_{2}\\ &\qquad+  \|\mathbf{x}^{(\text{mag})} \circ e^{\mathbbm{i}\phi}\mathbf{x}_{\text{est}}^{(\theta)} - \mathbf{x}_{\text{est}}^{(\text{mag})} \circ e^{i \phi} \mathbf{x}_{\text{est}}^{(\theta)}\|_{2}.\label{eqn: split up into phase and magnitude}
\end{align}
For the first term, we may use the inequality $\|\mathbf{u}\circ \mathbf{v}\|_2\leq \|\mathbf{u}\|_\infty \|\mathbf{v}\|_2$ to see that
\begin{align}\nonumber
\| \mathbf{x}^{(\text{mag})} \circ \mathbf{x}^{(\theta)} -  \mathbf{x}^{(\text{mag})} \circ e^{\mathbbm{i}\phi}\mathbf{x}_{\text{est}}^{(\theta)}\|_{2}
&\leq \|\mathbf{x}^{(\text{mag})}\|_{\infty} \|\mathbf{x}^{(\theta)}_{\text{est}} - e^{-\mathbbm{i}\phi}\mathbf{x}^{(\theta)}\|_{2}&\\\label{eqn: phase error bound}
&= \|\mathbf{x}\|_{\infty}  \|\mathbf{x}^{(\theta)}_{\text{est}} - e^{-\mathbbm{i}\phi}\mathbf{x}^{(\theta)}\|_{2}.
\end{align} For the second term, we see that
\begin{align}\nonumber
\|\mathbf{x}^{(\text{mag})} \circ e^{\mathbbm{i}\phi}\mathbf{x}_{\text{est}}^{(\theta)} - \mathbf{x}_{\text{est}}^{(\text{mag})} \circ e^{i \phi} \mathbf{x}_{\text{est}}^{(\theta)}\|_{2} 
  &\leq  \| e^{\mathbbm{i}\phi}\mathbf{x}_{\text{est}}^{(\theta)}\|_\infty \cdot \|\mathbf{x}^{(\text{mag})}  - \mathbf{x}_{\text{est}}^{(\text{mag})}\|_{2}&\\
&= \| \mathbf{x}^{(\text{mag})} -  \mathbf{x}^{(\text{mag})}_{\text{est}}\|_{2}\label{eqn: mag error bound}.
\end{align}
Combining \eqref{eqn: phase error bound} and \eqref{eqn: mag error bound} with \eqref{eqn: split up into phase and magnitude} and minimizing over $\phi$ completes the proof. 
\end{proof}

\section{Auxilliary Results from Spectral Graph Theory}\label{secB}
In this section, we will prove several lemmas related to the graph Laplacian and its eigenvalues.
The following definition defines a partial ordering on the set of weighted graphs induced by the spectrum of their graph Laplacians.
\begin{definition} We say that a symmetric matrix  $\mathbf{A}$ is positive semi-definite and write $\mathbf{A} \succeq \mathbf{0}$ if  $\mathbf{x}^T \mathbf{A}\mathbf{x} \geq 0,  \forall \mathbf{x} \in \mathbb{R}^n$ (or equivalently if all the eigenvalues of $\mathbf{A}$ are non-negative). We define the \textit{Loewner order}\footnotemark \footnotetext{The Loewner order is actually a partial ordering since there exist $\mathbf{A}$ and  $\mathbf{B}$ such that $\mathbf{A} \not\succeq \mathbf{B}$ and $\mathbf{B} \not\succeq \mathbf{A}$. 
} $\succeq$ by the rule that $\mathbf{A}\succeq \mathbf{B}$ if $\mathbf{A}-\mathbf{B}$ is positive semi-definite (or equivalently if  
$\mathbf{x}^T \mathbf{A} \mathbf{x} \geq \mathbf{x}^T\mathbf{B}\mathbf{x}, \forall \mathbf{x} \in \mathbb{R}^n$).
For two graphs $G$ and $H$ with the same number of vertices, we will define $G\succeq H$ if $\mathbf{L}_G\succeq\mathbf{L}_H.$ We will also write $G\succeq \sum_{i=0}^{n-1} H_i$ if $\mathbf{L}_G\succeq \sum_{i=0}^{n-1} \mathbf{L}_{H_i}$, and for a scalar $c$ we will write $G\succeq cH$ if $\mathbf{L}_G\succeq c\mathbf{L}_H$. 
\end{definition}
\begin{remark}\label{rem: spectral gap order}
If $G\succeq H$ and $\tau_G$ and $\tau_H$ are the smallest non-zero eigenvalues $\mathbf{L}_G$ and $\mathbf{L}_H$, then one can use the fact that $\tau_G = \underset{\substack{\mathbf{x} \in \mathbb{R}^n\\ \mathbf{x} \bot \mathbf{1}}}{\min} \dfrac{\mathbf{x}^T\mathbf{L}_G\mathbf{x}}{\mathbf{x}^T\mathbf{x}}$ (see \cite{sagt}) to verify that $\tau_G\geq \tau_H$. 
\end{remark}


We now define some basic terminology for weighted graphs. (We note that these definitions may also be applied to unweighted graphs by interpreting each edge as having weight one.)
\begin{definition} \label{def: basics weighted} \textbf{(Weighted Distance Definitions)} Let $G = (V,E,\mathbf{W})$ be a weighted graph.\\
(i) For any subgraph $H = (V',E')$ of $G$, we define the \textbf{weight} of $H$, denoted $w(H)$, as 
\begin{align*}
w(H) \coloneqq \sum_{(i,j)\in E'} W_{i,j},
\end{align*}  
(ii) If $P$ is a path inside $G$, we will let $\text{len}(P) \coloneqq w(P)$ denote the \textbf{weighted length} of $P$.\\
(iii) We define the \textbf{weighted distance} between two vertices $u$ and $v$,  $\text{dist}_G (u,v)$, to be the minimal weighted length of any path from $u$ to $v$\\
(iv) The \textbf{weighted diameter} of $G$, denoted by $\text{diam}(G)$, is the maximum distance between any two vertices in $G$, that is,
\begin{align*}
\text{diam}(G) \coloneqq \max\{\text{dist}_G (u,v) \mid (u,v) \in V \times V\}.
\end{align*}
\end{definition}
\noindent In some contexts, it will be useful to consider the pointwise inverses of the weights $W_{i,j}$.
\begin{definition} \textbf{(Inverse Weighted Distance Definitions)} Let $G = (V,E,\mathbf{W})$ be a weighted graph.\\
(i) For any subgraph  $H = (V',E')$ of $G$, the \textbf{inverse weight} of $H$, 
is defined by 
\begin{align*}
w^{-1}(H) \coloneqq \sum_{(i,j)\in E'} \dfrac{1}{W_{i,j}},
\end{align*}
(ii) For a path $P$ inside $G$, we refer to $\text{len}^{-1}(P) \coloneqq w^{-1}(P)$ as the \textbf{inverted weighted length} of $P$.\\
(iii) For two vertices $u$ and $v$ we will refer to the minimal value of $w^{-1}(P)$ over all paths from $u$ to $v$ as the \textbf{inverted weighted distance}, denoted by $\text{dist}^{-1}_{G} (u,v)$.\\
(iv) The \textbf{inverse weighted diameter} of $G$, denoted by $\text{diam}^{-1}(G)$, is the maximum distance between any two vertices in $G$, that is,
\begin{align*}
\text{diam}^{-1}(G) \coloneqq \max\{\text{dist}^{-1}_G (u,v) \mid (u,v) \in V \times V\}.
\end{align*}
\end{definition}

The proof of Lemma \ref{lem: spectral gap} (and thus Theorem \ref{thm:MainThm}), relies on the following  lemma to provide a lower bound for the spectral gap $\tau_{G}$. 

\begin{lemma}\label{Weighted Spectral Bound} (\textbf{Weighted Spectral Bound}) Let $G = (V,E, \mathbf{W})$ be a weighted, connected graph with $\vert V\vert = n$, and let $W_{\min}$ and $W_{\text{max}}$ denote the minimum and maximum value of any of the (nonzero) weights of $G$.  Then 
\begin{align*}
\tau_G \geq \dfrac{2 \cdot W_{\min}}{W_{\text{max}}(n-1) \cdot \text{diam}^{-1}(G)}.
\end{align*}
\end{lemma}
\noindent To prove Lemma \ref{Weighted Spectral Bound}, we recall the following lemma from \cite{sagt}. 

\begin{lemma}\label{Weighted Path Inequality} \textbf{(Weighted Path Inequality)} (Lemma 5.6.1 \cite{sagt}) Let $P_n=(v_0,v_1,\ldots,v_{n-1})$ be a path of length $n$ and assume that, for all $0\leq i < n-2$,  $w_i$, the weight of $(v_i,v_{i+1})$ is strictly positive. For $0\leq i <n-2,$ let $G_{i,i+1}=(V,(v_i,v_{i+1}))$ be the graph whose vertex set $V$ is that same as the vertex set of $G$ but only has a single edge $(v_i,v_{i+1})$. Similarly, let $G_{0,n-1} = (V,(v_0,v_{n-1}))$ be the graph with only a single edge $(v_0,v_{n-1})$.  Then
\begin{align*}
G_{0,n-1} \preccurlyeq \bigg( \sum_{i=0}^{n-2} \dfrac{1}{w_i}\bigg) \sum_{i=0}^{n-2}w_i G_{i,i+1}\footnotemark = \text{len}^{-1}({P_{n}}) \cdot P_{n},
\end{align*}
where the final equality is interpreted is the sense of  $A\preccurlyeq B$ and $B\preccurlyeq A$.
\footnotetext{Under this construction, we see that if we have a weight which is much larger than all of the others, it effectively gets nullified by taking the inverse.}
\end{lemma}
\begin{proof}[The Proof of Lemma \ref{Weighted Spectral Bound}] For $u,v\in V$, let $G_{u,v}=(V,(u,v))$ denote the graph with only a single edge from $u$ to $v$ and let $P_{u,v}$ denote a path from $u$ to $v$ with minimal weighted inverse length. Then, by Lemma \ref{Weighted Path Inequality} we have
\begin{align*}
G_{u,v} \preccurlyeq \text{len}^{-1}(P_{u,v}(G)) \cdot P_{u,v}(G) \preccurlyeq \text{diam}^{-1}(G) \cdot P_{u,v}(G) \preccurlyeq \text{diam}^{-1}(G) \cdot G,
\end{align*}
where the last inequality holds since for all subgraphs $H$ of a graph $G$, $H \preccurlyeq G$ (Section 5.2 \cite{sagt})

Let $\widetilde{K}_n$ be the extended weighted, complete graph on $n$ vertices with weighted matrix $\widetilde{\mathbf{W}}$, where $\widetilde{W}_{i,j} = 
\begin{cases}
W_{i,j}, &(i,j) \in E\cr
W_{\min}, &(i,j) \not\in E\cr
\end{cases}
$. Then by summing over all vertices, we have that
\begin{align*}
\mathbf{L}_{\widetilde{K}_n} = \sum_{0 \leq i < j \leq n-1} \widetilde{W}_{i,j}\mathbf{L}_{G_{i,j}} \preccurlyeq W_{\text{max}}\sum_{0 \leq i < j \leq n-1} \text{diam}^{-1}(G) \cdot \mathbf{L}_G,
\end{align*}
Since $\sum_{0\leq i<j\leq n-1} 1=n(n-1)$, 
we then have that
\begin{align*}
\widetilde{K}_n \preccurlyeq \dfrac{W_{\text{max}}n(n-1)}{2} \text{diam}^{-1}(G) \cdot G,
\end{align*} which, by Remark \eqref{rem: spectral gap order}  implies $$\tau_{\widetilde{K}_n} \leq  \dfrac{W_{\text{max}}n(n-1)}{2} \text{diam}^{-1}(G) \tau_G,$$
and therefore  $$\tau_G \geq \dfrac{2\tau_{\widetilde{K}_n}}{W_{\text{max}}n(n-1) \cdot \text{diam}^{-1}(G)}.$$
Letting $K_n$ be the unweighted graph on $n$ vertices we see that
\begin{align*}
\mathbf{x}^T \mathbf{L}_{\widetilde{K}_n} \mathbf{x} = \hspace{-3mm}\sum_{(a,b) \in [n]_0} \hspace{-3mm} \widetilde{W}_{a,b}(x(a) - x(b))^2 \geq W_{\text{min}} \hspace{-3mm}\sum_{(a,b) \in [n]_0} \hspace{-3mm} (x(a) - x(b))^2 = W_{\text{min}} \mathbf{x}^T\mathbf{L}_{K_n}\mathbf{x}.
\end{align*}
Therefore,
\begin{align*}
\tau_{\widetilde{K}_n} =  \underset{\substack{\mathbf{x} \in \mathbb{R}^n\\ \mathbf{x} \bot \mathbf{1}}}{\min} \dfrac{\mathbf{x}^T\mathbf{L}_{\widetilde{K}_n}\mathbf{x}}{\mathbf{x}^T\mathbf{x}} \geq W_{\text{min}} \underset{\substack{\mathbf{x} \in \mathbb{R}^n\\ \mathbf{x} \bot \mathbf{1}}}{\min} \dfrac{\mathbf{x}^T\mathbf{L}_{K_n}\mathbf{x}}{\mathbf{x}^T\mathbf{x}} \geq W_{\min} \cdot \tau_{K_n}.
\end{align*}
\text{Thus, since $\tau_{K_N} = n$ (5.4.1, \cite{sagt}), we have that $\tau_G \geq \dfrac{2W_{\text{min}}}{W_{\text{max}}(n-1) \cdot \text{diam}^{-1}(G)}$}.
\end{proof}
Our next result uses Lemma \ref{Weighted Spectral Bound} to produce a bound for $\tau_G$ in terms of the diameter of the underlying unweighted graph.
\begin{theorem} \label{General Weighted Spectral Gap} Let $G = (V,E,\mathbf{W})$ be a weighted graph and let $W_{\min}$ and $W_{\text{max}}$ be the minimum and maximum value of any its (nonzero) weights. Then 
\begin{align*}
\tau_{G}  \geq \dfrac{2 \cdot (W_{\text{min}})^2}{W_{\text{max}}(n-1)\text{diam}(G_{\text{unw}})},
\end{align*}
where $G_{\text{unw}}=(V,E)$ is the unweighted counterpart of $G$.
\end{theorem}

\begin{proof} Let $G' = (V,E,\mathbf{W}')$, where $W'_{i,j}=1/W_{i,j}$ if $W_{i,j}\neq 0$ and $W'_{i,j}=0$ otherwise.  Let $W'_{\max}$ be the maximum element of $\mathbf{W}'$. Observe that by construction, we have $W'_{\text{max}} = \dfrac{1}{W_{\min}}$. Moreover, it follows immediately from Definition  \ref{def: basics weighted} that we have $\text{diam}^{-1}(G)= \text{diam}(G')$. Therefore, 
\begin{align*}
\text{diam}^{-1}(G) = \text{diam}(G') \leq W'_{\max} \cdot \text{diam}(G_{\text{unw}}) = \dfrac{1}{W_{\min}} \text{diam}(G_{\text{unw}}).
\end{align*}

So by Lemma \ref{Weighted Spectral Bound}, we have that
\begin{align*}
\tau_G \geq \dfrac{2 \cdot W_{\min}}{W_{\text{max}}(n-1) \cdot \text{diam}^{-1}(G)} \geq \dfrac{2 \cdot (W_{\min})^2}{W_{\text{max}}(n-1)\text{diam}(G_{\text{unw}})}.
\end{align*}
\end{proof}



\end{document}